\documentclass{amsart}

\usepackage{graphicx,color}
\usepackage{geometry}

\title[Discretization for a fourth order diffusion equation]{Long-Time Behavior of a Finite Volume Discretization for a Fourth Order Diffusion Equation}
\author{Jan Maas}
\address{Jan Maas \\ Institute of Science and Technology Austria (IST Austria)
 \\ \newline Am Campus 1 \\  3400 Klosterneuburg \\ Austria}
\email{jan.maas@ist.ac.at}
\author{Daniel Matthes}
\address{Daniel Matthes \\ Technische Universit\"at M\"unchen \\ Zentrum Mathematik/M8 \\ Boltzmannstr. 3 \\ D-85747 Garching \\ Germany}
\email{matthes@ma.tum.de}
\thanks{This research was supported by the DFG Collaborative Research Centers TRR 109, ``Discretization in Geometry and Dynamics'' and 1060 ``The Mathematics of Emergent Effects''.}
\date{\today}

\newcommand{\setR}{{\mathbb R}}
\newcommand{\setRnn}{{\mathbb{R}_{\ge0}}}
\newcommand{\setN}{{\mathbb N}}
\newcommand{\qtext}[1]{\quad\text{#1}}
\newcommand{\qtextq}[1]{\quad\text{#1}\quad}
\newcommand{\dd}{\,\mathrm{d}}
\newcommand{\dn}{\mathrm{d}}
\newcommand{\dff}{\mathrm{D}}

\newcommand{\pot}{\Phi}
\newcommand{\aux}{\Psi}
\newcommand{\ons}{\mathbf{K}}
\newcommand{\cvx}{\mathbf{M}}
\newcommand{\ent}{\mathcal{H}}
\newcommand{\fish}{\mathcal{I}}

\newcommand{\mf}{\mathfrak{M}}
\newcommand{\tg}{\mathrm{T}}
\newcommand{\mc}{\mathbb{M}}

\newcommand{\cube}{\Omega}
\newcommand{\intcube}{\int_{\Omega}}
\newcommand{\prb}{\mathcal{P}_+}
\newcommand{\prbnn}{\mathcal{P}}
\newcommand{\bi}{{\mathbf i}}
\newcommand{\bj}{{\mathbf j}}
\newcommand{\nbh}[2]{#1\leftrightarrow#2}

\newcommand{\eps}{\varepsilon}

\newtheorem{thm}{Theorem}
\newtheorem{prp}{Proposition}
\newtheorem{lem}{Lemma}
\newtheorem{rmk}{Remark}

\DeclareMathOperator{\Hess}{Hess}

\newcommand{\ee}{\mathbf{e}}

\definecolor{darkred}{rgb}{0.9,0.1,0.1}

\begin{document}

\begin{abstract}
  We consider a non-standard finite-volume discretization 
  of a strongly non-linear fourth order diffusion equation on the $d$-dimensional cube, 
  for arbitrary $d\ge1$.
  The scheme preserves two important structural properties of the equation:
  the first is the interpretation as a gradient flow in a mass transportation metric,
  and the second is an intimate relation to a linear Fokker-Planck equation.
  Thanks to these structural properties, the scheme possesses two discrete Lyapunov functionals. 
  These functionals approximate the entropy and the Fisher information, respectively,
  and their dissipation rates converge to the optimal ones in the discrete-to-continuous limit.
  Using the dissipation, we derive estimates on the long-time asymptotics of the discrete solutions.
  Finally, we present results from numerical experiments 
  which indicate that our discretization is able to capture
  significant features of the complex original dynamics,
  even with a rather coarse spatial resolution.
\end{abstract}

\maketitle

\section{Introduction}
\subsection{The QDD equation}
In this note, we introduce and analyze a particular spatial discretization of the following non-linear parabolic equation of fourth order,
\begin{align}
  \label{eq:dlss}
  \partial_t u = 
 - \nabla\cdot\left(u\nabla\left[\frac{\Delta u}{u}+\Delta\log u\right]\right) + \nabla\cdot(u\nabla W),
  \quad u=u(t;x)>0,\,t>0,\,x\in\cube:=[0,1]^d,
\end{align}
subject to variational boundary conditions, see \eqref{eq:bc} below.
The potential $W:\cube\to\setR$ is assumed to satisfy certain structural conditions \eqref{eq:VW} and \eqref{eq:decompose};
a possible choice is $W(x)=\lambda^*/2|x-\bar x|^2$ for arbitrary $\lambda^*>0$ and $\bar x\in\setR^d$ 

Equation \eqref{eq:dlss}, which is referred to 
as \emph{Quantum drift diffusion (QDD) equation} or as \emph{Derrida-Lebowitz-Speer-Spohn (DLSS) equation} in the literature,
appears, e.g., in semi-conductor modelling \cite{Degond,Juengel} 
and in the analysis of interface motion in spin systems \cite{DLSS1,DLSS2}.
Depending on the context, the non-linear term is written in one of several equivalent forms:
\begin{align*}
  \nabla\cdot\left(u\nabla\left[\frac{\Delta u}{u}+\Delta\log u\right]\right)
  = 2\nabla\cdot\left(u\,\nabla\left(\frac{\Delta{\sqrt u}}{\sqrt u}\right)\right) 
  = \nabla^2:\big(u\,\nabla^2\log u\big).
\end{align*}
Existence and qualitative properties of (weak) solutions to \eqref{eq:dlss} have been intensively analyzed in the past two decades \cite{BLS,GST,JM,JT,JV,MMS}.
For instance, it has been proven --- see \cite{GST} for the most comprehensive result --- 
that the initial boundary value problem for \eqref{eq:dlss}\&\eqref{eq:bc}
possesses a non-negative and mass preserving global weak solution $u:\setR_+\times\Omega\to\setR$
for all non-negative initial conditions $u_0\in L^1(\Omega)$ of finite entropy.
By scaling invariance, we may assume without loss of generality in the following that the solution $u$ is a time-dependent probability density.

Several (semi-)discrete approximations of \eqref{eq:dlss} have been studied, both analytically and numerically. 
The schemes presented in \cite{BEJ,CJT,DMM,JP,MO} inherit \emph{some} structural properties of \eqref{eq:dlss}, 
like monotonicity of certain quantities. 
All of these schemes have in common that they provide non-negative (semi-)discrete solutions.

Here we continue in the spirit of \cite{MO},
where a discretization was performed on grounds of \eqref{eq:dlss}'s gradient flow structure with respect to the $L^2$-Wasserstein metric, which leads to a scheme that simultaneously preserves two essential Lyapunov functionals.
From these Lyapunov functionals, estimates on the fully discrete solutions were derived
and have been used to analyze their long-time asymptotics \cite{O} and the discrete-to-continuous limit \cite{MO}.

However, here we do not use the Lagrangian structure behind \eqref{eq:dlss} --- which was essential in \cite{MO,O} ---
but define a scheme on grounds of a finite-volume discretization.
Our ansatz is motivated by a particular structure-preserving discretization of linear Fokker-Planck equations,
which has been introduced simultaneously 
in \cite{CHLZ11,Maas,Mielke11a}.
Using this ``Eulerian approach'', we overcome the limitation of \cite{MO,O} to $d=1$ space dimension.
The similarities with \cite{MO,O} are that we rely on the gradient flow formulation of \eqref{eq:dlss},
and that we design the discretization in such a way that enforces monotonicity of two Lyapunov functionals.
We remark that the general idea to preserve simultaneous monotonicity of several functionals in the discretization
has been used for other equations before, 
like in the context of the formally similar thin film equations, see \cite{Grun,GrunRumpf,Bertozzi}.

\subsection{Structural properties and long-time asymptotics}
Most of the qualitative results for \eqref{eq:dlss} are based on two fundamental structural properties:
the first is its gradient flow structure with respect to the $L^2$-Wasserstein metric \cite{GST},
and the second is an intimate relation to a certain Fokker-Planck equation \cite{DM}.
That Fokker-Planck equation has the form
\begin{align}
  \label{eq:fokker}
  \partial_s v_s = \Delta v_s + \nabla\cdot(v_s\nabla V) \qtextq{in $\cube$,}
  \qquad
  \partial_\nu(v_s/\pi)=0 \qtext{on $\partial\cube$},
\end{align}
where $\pi:\Omega\to\setR_+$ is given by
\begin{align}
  \label{eq:steady}
  \pi(x) = \frac1Ze^{-V(x)} \qtextq{with} Z=\intcube e^{-V(x')}\dd x'
\end{align}
and defines the unique stationary probability density $\pi$ for \eqref{eq:fokker}.
To establish the connection between \eqref{eq:dlss} and \eqref{eq:fokker}, 
we shall assume henceforth that the respective potentials $V$ and $W$ are related via
\begin{align}
  \label{eq:VW}
  W=|\nabla V|^2-2\Delta V,
\end{align}
and that $V$ is $\lambda$-convex with some positive $\lambda$, i.e., $\nabla^2V\ge\lambda>0$. 
Notice that $V(x)=\frac\lambda2|x-\bar x|^2$ is an admissible choice, 
and leads to $W(x)=\lambda^2|x-\bar x|^2-2d\lambda$.

A direct computation shows that $\pi$ is a stationary solution to \eqref{eq:dlss} as well,
provided the boundary conditions are chosen appropriately:
\begin{align}
  \label{eq:bc}
  \partial_\nu(u/\pi) = \partial_\nu \left(\frac{\Delta u}{u}+\Delta\log u+W\right)=0 \quad\text{on $\partial\cube$}.
\end{align}
Another formal computation reveals that \eqref{eq:dlss} and \eqref{eq:fokker} have two Lyapunov functionals in common,
namely the relative logarithmic entropy $\ent$ and the relative Fisher information $\fish$,
given by
\begin{align}
  \label{eq:entfish}
  \ent(w) = \intcube w\log(w/\pi)\dd x \qtextq{and} \fish(w) = \intcube w\,|\nabla\log(w/\pi)|^2\dd x.
\end{align}
In fact, both \eqref{eq:fokker} and \eqref{eq:dlss} are gradient flows
--- for $\ent$ and for $\fish$, respectively --- 
in the $L^2$-Wasserstein metric.
That is, formally, we can write 
\begin{align}
  \label{eq:ogf}
  \partial_sv = -\ons_v\dff_v\ent
  \qtextq{and}
  \partial_tu = -\ons_u\dff_u\fish,
\end{align}
respectively,
where $\ons$ is the Onsager operator (inverse metric tensor) of the Wasserstein metric,
\begin{align*}
  \ons_u\xi = - \nabla\cdot(u\nabla\xi).
\end{align*}
The final but most important connection between \eqref{eq:fokker} and \eqref{eq:dlss} is 
the following relation between the respective potentials of the two gradient flows:
\begin{align}
  \label{eq:miracle}
  \fish(v) = \dff_v\ent[\ons_v\dff_v\ent].
\end{align}
That is, the potential of the gradient flow \eqref{eq:dlss} is the dissipation of the entropy $\ent$ along its own gradient flow.
Despite the fact that the representation \eqref{eq:miracle} of $\fish$ is classical, 
implications on the dynamics of the fourth order equation \eqref{eq:dlss} have been drawn only recently in \cite{DM}, 
see also \cite{CT,MMS,MS}.

It turns out \cite{DM,MMS} that the equilibration behavior of \eqref{eq:dlss} 
is intimately related to the one of \eqref{eq:fokker}.
We summarize the relevant estimates.
Thanks to the $\lambda$-convexity of $V$, it follows 
that both $\ent$ and $\fish$ decay with exponential rate $\lambda$ along solutions $v$ to \eqref{eq:fokker},
\begin{align}
  \label{eq:fokkerdecay}
  \ent(v_s)\le \ent(v_{s'}) e^{-2\lambda(s-s')} \qtextq{and} \fish(v_s)\le\fish(v_{s'}) e^{-2\lambda(s-s')} \qtext{for all $s\ge s'\ge0$},
\end{align}
and that the Fisher information can be estimated just in terms of the initial value of the entropy,
\begin{align}
  \label{eq:fokkerauxiliary}
  \fish(v_s) \le \ent(v_0)s^{-1} \qtext{for all $s>0$.}
\end{align}
With $V$ and $W$ related by \eqref{eq:VW}, 
the following analogous estimate can be shown for solutions to the QDD equation \eqref{eq:dlss}:
\begin{align}
  \label{eq:dlssdecay}
  &\ent(u_t)\le\ent(u_{t'}) e^{-(2\lambda)^2(t-t')} 
  \qtextq{and} 
  \fish(u_t)\le\fish(u_{t'}) e^{-(2\lambda)^2(t-t')} \qtext{for all $t\ge t'\ge0$,} \\
  \label{eq:auxiliary}
  &\fish(u_t)\le\ent(u_0)(2\lambda t)^{-1} \qtext{for all $t>0$.}
\end{align}
We review the derivation of \eqref{eq:fokkerdecay}--\eqref{eq:auxiliary} in Section \ref{sct:formal}.

\subsection{Discretization and main result}
The leading principle for our spatial discretization of \eqref{eq:dlss} is that 
the semi-discrete solutions to that scheme inherit the estimates in \eqref{eq:dlssdecay} and \eqref{eq:auxiliary}.
We discretize \eqref{eq:dlss} and \eqref{eq:fokker} simultaneously in order to preserve their close relation.

For the discretization of \eqref{eq:fokker} we follow an approach based on the entropy gradient flow structure for Markov chains developed in \cite{CHLZ11,Maas,Mielke11a}, which has been subsequently applied in \cite{DisLie,EM13,GM12,Mielke}.

We perform a finite volume discretization with a regular cubic lattice:
fix a box length $h=1/N$ with $N\in\setN$
and consider piecewise constant probability densities $u^h$ 
on the equi-distant subdivision of $\cube$ in $N^d$ sub-cubes of side length $h$.
Now, we replace \eqref{eq:ogf} by
\begin{align}
  \label{eq:dgf}
  \partial_sv^h = -\ons^h_v\dff_v\ent^h
  \qtextq{and}
  \partial_tu^h = -\ons^h_u\dff_u\fish^h,
\end{align}
respectively, where the discretized entropy $\ent^h$ is given
(up to an additive constant $\gamma^h>0$ defined in \eqref{eq:gammah}) by the restriction of $\ent$,
and the discretized Fisher information $\fish^h$ is obtained by the relation \eqref{eq:miracle}, i.e.,
\begin{align}
  \label{eq:dmiracle}
  \ent^h(u^h) = \ent(u^h)-\gamma^h, \quad \fish^h(u^h) = \dff_{u^h}\ent^h[\ons_{u^h}\dff_{u^h}\ent^h].
\end{align}
The discretized Onsager operator $\ons^h$
--- which implicitly determines a metric on the piecewise constant density functions ---
is designed such that the gradient flow of $\ent^h$ is the forward equation for a continuous time Markov chain.
The appropriate and rather non-obvious choice for $\ons^h$, see \eqref{eq:dons}, 
was independently found in \cite{Maas} and in \cite{Mielke11a}.

For the main result that we formulate below we need an additional hypothesis 
on the potential $V$, namely that
\begin{align}\label{eq:decompose}
  V(x) = V^{[1]}(x_1) + \cdots + V^{[d]}(x_d) 
\end{align}
for suitable functions $V^{[k]}:[0,1]\to\setR$, which, by definition of $\pi$ in \eqref{eq:steady},
is equivalent to the following factorization of the steady state:
\begin{align}
  \label{eq:factorize}
  \pi(x) = \pi^{[1]}(x_1)\cdots\pi^{[d]}(x_d), 
  \qtextq{where}\pi^{[k]}(x)=\frac1{Z^{[k]}}e^{-V^{[k]}(x)},
\end{align}
with suitable normalization constants $Z^{[k]}>0$ such that
the $\pi^{[1]},\ldots,\pi^{[d]}$ are probability densities on $[0,1]$.
Under the discretization, $\pi$ is replaced by a particular piecewise constant approximation $\pi^h$,
which is the unique minimizer of $\ent^h$, see Lemma \ref{lem:dpi}. 
The approximation $\pi^h$ still factors in the same form as above, see \eqref{eq:dfactor}.
\begin{thm}
  \label{thm:main}
  Assume that a pair of potentials $V$, $W$ satisfying 
  the relation \eqref{eq:VW} and the technical hypothesis \eqref{eq:decompose} is given.
  Assume further that $V$ is $\lambda$-convex with some $\lambda>0$.

  For a given discretization parameter $h>0$, 
  define discretized entropy and Fisher information as in \eqref{eq:dmiracle}, 
  and a discrete Onsager operator as in \eqref{eq:dons}.
  Then any solution $u^h$ of the discrete gradient flow
  \begin{align*}
    \partial_tu^h = -\ons^h_{u^h}\dff_{u^h}\fish^h    
  \end{align*}
  satisfies the following analogues of \eqref{eq:dlssdecay} and \eqref{eq:auxiliary},
  \begin{align}
    \label{eq:ddlssdecay}
    &\ent^h(u^h_t)\le \ent^h(u^h_{t'}) e^{-(2\lambda^h)^2(t-t')} 
    \qtextq{and} 
    \fish^h(u^h_t)\le\fish^h(u^h_{t'}) e^{-(2\lambda^h)^2(t-t')} \qtext{for all $t\ge t'\ge0$,}\\
    \label{eq:dauxiliary}
    &\fish^h(u^h_t) \le \ent^h(u_0)\,(2\lambda^ht)^{-1} \qtext{for all $t>0$.}
  \end{align}
  Consequently, $u^h_t$ approaches the equilibrium $\pi^h$ exponentially fast,
  \begin{align}
    \label{eq:L1}
    \|u^h_t-\pi^h\|_{L^1(\cube)} \le \sqrt{2\ent^h(u^h_0)}\,e^{-2(\lambda^h)^2t}.
  \end{align}
  Above, $\lambda^h=\lambda+O(h^2)$ as $h\downarrow0$.
\end{thm}

\subsection{Geodesic convexity vs. convex entropy decay}
All of the --- continuous and discrete --- equations under consideration here 
will be gradient flows of geodesically $\lambda$-convex functionals.
The proof of our main result Theorem \ref{thm:main} above, however, 
does not require to use the full power of $\lambda$-convexity.
Instead, we work with a weaker property that we call \emph{convex decay inequality},
see \eqref{eq:cvx-e-d} in Section \ref{sct:flowestimates}.
In a nutshell, the difference is that 
we do not require the Hessian of the functional to be larger or equal to $\lambda$ in \emph{every} direction,
but only in the direction of the functional's own gradient, at each given point.

This weaker form of convexity has been used in numerous places and in various disguises 
for the derivation of equilibration estimates, typically in connection with the Bakry-\'{E}mery method, 
see e.g. \cite{frech} and references therein.
Recently, an adaptation of this convexity concept to Markov chains has been developed in \cite{CDPP09}.
There are examples --- see Remark \ref{rmk:MielkevsCaputo} ---
where the modulus of convexity improves (slightly) 
upon relaxation from geodesic convexity to convex decay. 

A key technical ingredient in the proof of our main result is the tensorization property of the convex decay inequality.
This result is given in Section \ref{sct:tensor}, and might be of independent interest.

\subsection{Plan of the paper}
In Section \ref{sct:classical} below, we review 
the basic results from the general theory of gradient flows and the Bakry-\'{E}mery method 
which are relevant for the study of our equation \eqref{eq:dlss} and our discretization.
Sections \ref{sct:dfokker} and \ref{sct:ddlss} are devoted to discretizations.
In Section \ref{sct:dfokker}, we analyze the properties of a finite-volume discretization 
for the linear Fokker-Planck equation \eqref{eq:fokker} in the spirit of \cite{Maas,Mielke11a}.
In Section \ref{sct:ddlss}, we define a ``compatible'' discretization of the QDD equation and
prove the main result Theorem \ref{thm:main}.
We conclude by discretizing in time as well, 
and perfoming a series of numerical experiments in dimension $d=2$
to illustrate the (non-)optimality of the theoretical decay estimates.


\section{Estimates for $\lambda$-convex gradient flows}
\label{sct:classical}
In this section, we shall mainly collect and rephrase classical and recent results 
about the large-time behavior of gradient flows.
Throughout this section, we assume smoothness of all appearing analytical structures.
These smoothness assumptions are justified 
in the analysis of the discretizations in Sections \ref{sct:dfokker}\&\ref{sct:ddlss} below, 
provided that one restricts to strictly positive probability densities.
The application to solutions of the original evolution equation \eqref{eq:dlss}, however, 
are purely formal and only serve as a motivation.

\subsection{$\lambda$-convex gradient flows}
Let a smooth Riemannian manifold $\mf$ with metric $d$ be given.
For simplicity we assume that $\mf$ is an open subset of a (finite-dimensional) affine space $\mathcal{X}$.
At each point $u\in\mf$, there is a one-to-one correspondence 
between the scalar product $\langle\cdot,\cdot\rangle_u$ on $\tg_u\mf$
and the Onsager operator $\ons_u:\tg_u^\star\mf\to\tg_u\mf$,
which is the uniquely determined linear isomorphism with
\begin{align*}
 \langle \ons_up,\xi \rangle_u = p[\xi] \qtext{for all $\xi\in\tg_u\mf$ and $p\in\tg_u^\star\mf$.}
\end{align*}
Note that $\ons_u$ is symmetric, in the sense that 
\begin{align*}
 p_1[\ons_u p_2]
   = \langle \ons_u p_1, \ons_u p_2 \rangle_u
   = \langle \ons_u p_2, \ons_u p_1 \rangle_u
   =  p_2[\ons_u p_1]
   \qtext{for all $p_1, p_2 \in\tg_u^\star\mf$.}
\end{align*}
In the application discussed here, the Onsager operator (and not the scalar product) will be the given quantity. In fact, in our application, the Onsager operator extends continuously to the boundary of $\mf$ in $\mathcal{X}$, while the Riemannian metric degenerates at the boundary.

The gradient flow of a given smooth potential $\pot:\mf\to\setR$
is then defined as (solution to) the differential equation
\begin{align}
  \label{eq:gf}
  \dot u = F_\pot(u) := -\ons_u\dff_u\pot.
\end{align}
By smoothness of $\pot$, 
local solutions $u:[0,T)\to\mf$ to \eqref{eq:gf} exist for any initial condition $u_0\in\mf$, 
and the only possible obstruction to global existence is that $u$ leaves $\mf$ at time $T>0$.

A central notion in the theory is that of $\lambda$-convexity of $\pot$ (with $\lambda \in \setR$),
which means that $\Hess \pot \geq \lambda$, where the Hessian is to be understood in the Riemannian structure of the $\mf$, and the inequality holds in the sense of quadratic forms:
\begin{align*}
 \Hess_u \pot[\xi, \xi] \geq \lambda \| \xi \|_u^2 \qtext{for all $u \in \mf$ and $\xi \in \tg_u\mf$.}
\end{align*}
An elegant ``Eulerian'' approach for proving $\lambda$-convexity has been developed in \cite{OW,DS}. 
This approach has been implemented in \cite{EM} on the manifold of probability measures over a finite state space. 
A useful characterization of $\lambda$-convexity, that does not involve the metric but only the Onsager operator, has been formulated in \cite{LM}: 
at each $u\in\mf$, define the bi-linear form $\cvx_u$ on $\tg^\star_u\mf$ via\begin{align*}
  \cvx_u[p,p] = - p\big[\dff_uF_\pot [\ons_up]\big] + \frac12 p\big[\dff_u\ons[F_\pot(u)]p\big].
\end{align*}
Then the functional $\pot$ is $\lambda$-convex if and only if the tensor $\cvx$ satisfies the estimate
\begin{align}
  \label{eq:cvx}
  \cvx\ge\lambda\ons
\end{align}
in the sense that $\cvx_u[p,p]\ge\lambda p[\ons_up]$ for all $u\in\mf$ and $p\in\tg_u^\star\mf$.
Here the differential $\dff_u$ is to be interpreted using the linear structure of the ambient space $\mathcal{X}$:
\begin{align*}
 \dff_u F_\pot[\xi] = \lim_{\eps \to 0} \frac{F_\pot(u + \eps \xi) - F_\pot(u)}{\eps} \in \tg_u \mf,
\end{align*}
and analogously for $\dff_u\ons$.
\begin{rmk}
In a smooth Riemannian setting,  $\lambda$-convexity of $\pot$ implies $\lambda$-contractivity for its gradient flow, i.e.,
\[d(u_t,u_t')\le e^{-\lambda t}d(u_0,u_0') 
\qtext{for all $t>0$ and arbitrary solutions $u$, $u'$ to \eqref{eq:gf}}.\]
\end{rmk}

\subsection{Estimates on the flow}
\label{sct:flowestimates}
In the following discussion, we will \emph{not} require the full strength of the $\lambda$-convexity assumption $\cvx \geq \lambda \ons$ from \eqref{eq:cvx}. 
Instead, in our calculations we will only apply \eqref{eq:cvx} to the argument $\dff_v\pot$.
The resulting \emph{convex decay inequality}
\begin{align}\label{eq:cvx-e-d}
  \tag{CDI} 
  \cvx[\dff_v\pot,\dff_v\pot] \ge \lambda\dff_v\pot[\ons_v\dff_v\pot]
\end{align}
is weaker than \eqref{eq:cvx}. 
Since 
\[ -\frac{\dn}{\dd s}\Phi(v_s) = \dff_v\pot[\ons_v\dff_v\pot]
\quad\text{and} \quad
\frac12 \frac{\dn^2}{\dd^2s}\Phi(v_s) = \cvx(\dff_v\pot,\dff_v\pot) \] 
--- as will be shown in the proof below ---
\eqref{eq:cvx-e-d} provides a relation between the first and the second derivative of $\pot$ along its gradient flow.
The inequality \eqref{eq:cvx-e-d} lies at the heart of the Bakry-\'Emery approach to functional inequalities \cite{BE}. 
In a Markov chain setting, the inequality \eqref{eq:cvx-e-d} has been studied in \cite{CDPP09}, see also \cite{EMT14,FM15}.

Our general hypothesis in the remainder of this section is that \eqref{eq:cvx-e-d} holds for some $\lambda>0$. 
We also assume that $\pot$ has a unique global minimizer $\bar u \in \mf$, 
and, without loss of generality,  that $\pot(u)\ge\pot(\bar u)=0$ for all $u\in\mf$.  
Clearly, these conditions are satisfied when $\pot$ is $\lambda$-convex, in which case $\bar u$ is the only critical point of $\pot$.

The auto-dissipation $|\partial\pot|^2:\mf\to\setR$ of $\pot$ is defined by
\begin{align}
  \label{eq:1}
  |\partial\pot|^2(u) = \dff_u\pot[\ons_u\dff_u\pot].
\end{align}
It follows from our assumptions that $|\partial\pot|^2(\bar u)=0$.
\begin{prp}[Gradient flow estimates for $\Phi$]
  \label{prp:be}
  Along any solution $(v_s)_{s\ge0}$ of the gradient flow \eqref{eq:gf}, 
  we have, for arbitrary $s\ge s'\ge0$,
  \begin{align}
    \label{eq:entdown}
    \pot(v_s) &\le \pot(v_{s'})e^{-2\lambda(s-s')}, \\
    \label{eq:fishdown}
    |\partial\pot|^2(v_s) &\le |\partial\pot|^2(v_{s'})e^{-2\lambda(s-s')},
  \end{align}
  and further, for arbitrary $s>0$,
  \begin{align}
    \label{eq:fishbyent}
    |\partial\pot|^2(v_s) 
    \le\pot(v_0)\, \frac{2\lambda}{e^{2\lambda s} - 1} 
 \le\pot(v_0)\,s^{-1}.
  \end{align}
  Moreover, the following functional inequality holds for arbitrary $v\in\mf$:
  \begin{align}
    \label{eq:logsob}
    2\lambda\pot(v)\le|\partial\pot|^2(v).
  \end{align}
\end{prp}
These results are classical.
We sketch a proof here, which derives the estimates directly from the  hypothesis \eqref{eq:cvx-e-d} by elementary calculations.
The idea is to use the method of iterated gradients from \cite{BE}.
\begin{proof}
  We start by proving \eqref{eq:fishdown}.
  To this end, we estimate the decay  in time of
  \begin{align}
    \label{eq:auxstrange}
    |\partial\pot|^2(v) = \dff_v\pot[\ons_v\dff_v\pot]
      = \sum_{i,j} \ons_{ij}(v) \partial_i \pot(v)  \partial_j \pot(v),
  \end{align}
which is
\begin{equation}\begin{aligned}    \label{eq:J}
    J(v_s) := -\frac{\dd}{\dd s}|\partial\pot|^2(v_s).
\end{aligned}\end{equation}  
From the last representation in \eqref{eq:auxstrange},
we obtain, writing $\pot_i = \partial_i \pot$ and $\pot_{ij} = \partial_i \partial_j\pot$ for brevity,
\begin{align*}
  J(v)  &= - \sum_{i,j,k} \Big(2\ons_{ij}(v)  \pot_{ik}(v) 
  +  \partial_k \ons_{ij}(v) \pot_i(v)     \Big) \pot_j(v) F_k^\pot(v).
\end{align*}  
On the other hand, the definition of $\cvx_v$ yields
\begin{align*}
  & \cvx_v[\dff_v\pot,\dff_v\pot]
  =  \dff_v\pot\big[\dff_v F_\pot [F_\pot(v)]\big] +  \frac12 \dff_v\pot\big[\dff_v\ons[F_\pot(v)]\dff_v\pot\big] \\ 
  & =  \sum_{i,j}\pot_i(v) \partial_j F_i^\pot(v) F_j^\pot(v) 
  + \frac12 \sum_{i,j,k}  \pot_i(v) \partial_k \ons_{ij}(v) F_k^\pot(v) \pot_j(v)  \\
  &  = - \sum_{i,j,k} \pot_i(v) \Big(  \partial_j \ons_{ik}(v) \pot_k(v) +  \ons_{ik}(v) \pot_{j k}(v) \Big)  F_j^\pot(v)
  - \frac12 \pot_i(v) \partial_k \ons_{ij}(v) F_k^\pot(v)  \pot_j(v)  \\
  &  = - \sum_{i,j,k}  \pot_i(v) \Big( \frac12 \partial_j \ons_{ik}(v)  \pot_k(v) +  \ons_{ik}(v) \pot_{jk}(v) \Big)  F_j^\pot(v),
\end{align*}
where the last identity follows by relabeling the indices. We thus obtain the crucial identity
\begin{align}\label{eq:J-M}
  J(v) = 2 \cvx_v[\dff_v\pot,\dff_v\pot].
\end{align}
Applying the assumption \eqref{eq:cvx-e-d}, we infer that
\begin{align}\label{eq:HF}
 J(v) \geq 2\lambda\dff_v\pot[\ons_v\dff_v\pot] = 2\lambda|\partial\pot|^2(v).
\end{align}
  Now apply Gronwall's lemma to the resulting inequality 
  \[ -\frac{\dd}{\dd s}|\partial\pot|^2(v_s) \ge 2\lambda|\partial\pot|^2(v_s)\]
  to obtain \eqref{eq:fishdown}.
  Next, we verify the functional inequality \eqref{eq:logsob}.
  First, observe that
  \begin{align}
    \label{eq:help001}
    -\frac{\dd}{\dd s}\pot(v_s) = \dff_{v_s}\pot[-F_\pot(v_s)] = |\partial\pot|^2(v_s).
  \end{align}
  For any fixed $s>s'>0$, this allows to conclude that
  \begin{align*}
    \pot(v_{s'})-\pot(v_s) = \int_{s'}^s |\partial\pot|^2(v_r)\dd r
    \le |\partial\pot|^2(v_{s'})\int_{s'}^s e^{-2\lambda(r-s')}\dd r,
  \end{align*}
  which implies that
  \begin{align*}
    2\lambda(\pot(v_{s'})-\pot(v_s)) \le (1-e^{-2\lambda(s-s')})|\partial\pot|^2(v_{s'}).    
  \end{align*}
  In the limit $s\to\infty$, we have $\pot(u_s)\to\pot(\bar u)=0$, 
  and thus we end up with
  \begin{align*}
    2\lambda\pot(u_{s'}) \le |\partial\pot|^2(u_{s'}),
  \end{align*}
  which verifies \eqref{eq:logsob}.
  To prove \eqref{eq:entdown},
  simply combine \eqref{eq:help001} with \eqref{eq:logsob} and apply Gronwall's lemma again.
  Finally, the estimate \eqref{eq:fishbyent} is a consequence of the following calculation,
  using that $s\mapsto|\partial\pot|^2(v_s)$ is a monotone function thanks to \eqref{eq:fishdown}:
  \begin{align*}
    \frac{e^{2\lambda s}-1}{2\lambda}|\partial\pot|^2(v_s) 
  &  = \int_0^s e^{2\lambda(s - s')}\dd s'\,|\partial\pot|^2(v_s)
    \le \int_0^s |\partial\pot|^2(v_{s'})\dd s'
 \\&   = -\int_0^s \frac{\dd}{\dd r}\bigg|_{r=s'}\pot(v_r)\dd s'
    = \pot(v_0)-\pot(v_s).
  \end{align*}
  By non-negativity of $\pot$, we arrive at \eqref{eq:fishbyent}.
\end{proof}

\subsection{Estimates on the flow of the dissipation functional}
We continue to assume that \eqref{eq:cvx-e-d} holds with some $\lambda>0$.
We also assume the normalization $\pot(\bar u)=0$, with $\bar u$ being the global minimizer.
Below, we study another gradient flow, namely the one generated by the dissipation $\aux=|\partial\pot|^2$,
\begin{align}
  \label{eq:agf}
  \dot u = F_\aux(u) = -\ons_u\dff_u\aux .
\end{align}
In general, no information is available on the convexity of the flow induced by $F_\aux$.
Still, the following analogue of Proposition \ref{prp:be} holds, 
thanks to the intimate relation of $\aux$ to the $\lambda$-convex functional $\pot$.
\begin{prp}[Gradient flow estimates for $\Psi$]
  \label{prp:mms}
  Along any solution $(u_t)_{t>0}$ of the auxiliary gradient flow \eqref{eq:agf}, 
  we have, for arbitrary $t\ge t'\ge0$,
  \begin{align}
    \label{eq:entddown}
    \pot(u_t) &\le \pot(u_{t'})e^{-(2\lambda)^2(t-t')}, \\
    \label{eq:fishddown}
    |\partial\pot|^2(u_t) &\le |\partial\pot|^2(u_{t'})e^{-(2\lambda)^2(t-t')},
  \end{align}  
  and further, for arbitrary $t>0$,
  \begin{align}
    \label{eq:fishbyent2}
    |\partial\pot|^2(u_t)
    \le \pot(u_0)\,\frac{2\lambda}{e^{(2\lambda)^2 t} - 1}
    \le\pot(u_0)\,(2\lambda t)^{-1}.
  \end{align}
\end{prp}
This result has been proven in the setting of metric spaces in \cite[Section 3]{MMS}.
As for Proposition \ref{prp:be}, we sketch a proof here
which only uses the inequality \eqref{eq:cvx-e-d} and some elementary calculations.
\begin{proof}
  We start by estimating the decay of $\pot(u_t)$ in time, i.e.,
  \begin{align*}
    J(u_t):=-\frac{\dd}{\dd t}\pot(u_t).
  \end{align*}
  Observe that, thanks to the symmetry of the Onsager operator,
  \begin{align*}
    J(u) = \dff_u\pot[-F_\aux(u)] 
    = \dff_u\pot\big[\ons_u\dff_u\aux\big]
    = \dff_u\aux\big[\ons_u\dff_u\pot\big]
    = \dff_u\aux\big[-F_\pot(u)],
  \end{align*}
  thus the $J$ defined above coincides with the $J$ defined in \eqref{eq:J}.
  From the inequality \eqref{eq:HF} in combination with the inequality \eqref{eq:logsob},
  it follows that
  \begin{align}
    \label{eq:interm}
    J(u) \ge 2\lambda \aux(u) \ge (2\lambda)^2 \pot(u).
  \end{align}
  Another application of Gronwall's lemma yields \eqref{eq:entddown}.
  In preparation for the proof of \eqref{eq:fishddown}, 
  observe that the Cauchy-Schwarz inequality for the scalar product $\langle\cdot,\cdot\rangle_u$
  translates into the following inequality for the Onsager operator $\ons_u$:
  \begin{align*}
    (p[\ons_uq])^2 \le p[\ons_up]\,q[\ons_uq] \qtext{for all $p,q\in\tg_u^\star\mf$.}
  \end{align*}
  In combination with the estimate \eqref{eq:HF}, we obtain
  \begin{align*}
    (2\lambda)^2\big(\dff_u\pot[\ons_u\dff_u\pot]\big)^2
    \le J(u)^2 = \big(\dff_u\aux[\ons_u\dff_u\pot]\big)^2
    \le \dff_u\aux[\ons_u\dff\aux]\,\dff_u\pot[\ons_u\dff_u\pot].
  \end{align*}
  Division by $\dff_u\pot[\ons_u\dff_u\pot]$ leads to
  \begin{align}\label{eq:interm-2}
    I(u):=\dff_u\aux[\ons_u\dff_u\aux]
    \ge (2\lambda)^2\dff_u\pot[\ons_u\dff_u\pot]
    = (2\lambda)^2\aux(u).
  \end{align}
  Since $I(u_t) = -\frac{\dd}{\dd t}\aux(u_t)$, we obtain \eqref{eq:fishddown} by yet another application of 
  Gronwall's lemma.
  For the proof of \eqref{eq:fishbyent2}, 
  we use the inequality \eqref{eq:fishddown} 
  and the first inequality from \eqref{eq:interm}.
  We thus obtain
  \begin{align*}
    \frac{e^{(2\lambda)^2 t}-1}{(2\lambda)^2}|\partial\pot|^2(u_t) 
    &= \int_0^t e^{(2\lambda)^2(t - t')} |\partial\pot|^2(u_t) \dd t'
    \leq \int_0^t  |\partial\pot|^2(u_{t'}) \dd t'
    \\
    &\le \frac1{2\lambda} \int_0^t J(u_{t'})\dd t'
    = -\frac1{2\lambda} \int_0^t \frac{\dd}{\dd r}\bigg|_{r=t'}\pot(u_r)\dd t'
    = \frac{\pot(u_0)-\pot(u_t)}{2\lambda},
  \end{align*}
  from which the first inequality in \eqref{eq:fishbyent2} follows since $\pot$ is non-negative. The second inequality is elementary.
\end{proof}

\subsection{Application: asymptotics for the Fokker-Planck and QDD equation}
\label{sct:formal}
To conclude our short review on gradient flows with \eqref{eq:cvx-e-d}, we show how the the estimates \eqref{eq:fokkerdecay}--\eqref{eq:auxiliary} on the long-time asymptotics
for solutions to \eqref{eq:fokker} and \eqref{eq:dlss}, respectively, can be obtained from Propositions \ref{prp:be} and \ref{prp:mms} above,
at least formally.
For the \emph{rigorous} derivation of the stated long-time asymptotics by variational methods,
we refer the reader to \cite{AGS} and to \cite{MMS}.

We consider the set $\prb(\cube)$ of strictly positive probability densities $u:\cube\to\setR_+$,
endowed with the $L^2$-Wasserstein metric, as Riemannian manifold $\mf$.
Tangent and cotangent vectors at $u\in\mf$ are identified with functions $\xi,p\in L^2(\cube)$ of vanishing mean,
their pairing being given by
\[ p[\xi] = \intcube p(x)\xi(x)\dd x. \]
The definition of the scalar product on the tangent spaces is intricate 
(it requires the solution of an auxiliary elliptic problem),
but the associated Onsager operator $\ons_u$ has an explicit form:
\begin{align}
  \label{eq:ons}
  \ons_up = -\nabla\cdot(u\nabla p).
\end{align}
In this framework, the Fokker-Planck equation \eqref{eq:fokker} can be written
as the gradient flow of the entropy $\ent$ from \eqref{eq:entfish}:
\begin{align}
  \label{eq:KDH}
  \partial_sv_s = \Delta_\pi v_s
  \qtextq{with}
  \Delta_\pi v = -\ons_v\dff_v\ent = \nabla\cdot\big(v\nabla\log(v/\pi)\big) = \Delta v + \nabla\cdot(v\nabla V).
\end{align}
This representation has been the starting point for the existence proof in the celebrated work \cite{JKO}.

Next, by the results of McCann \cite{McC},
the $\lambda$-convexity of the potential $V$ implies $\lambda$-convexity of this gradient flow;
see also \cite{DS} for an alternative proof of this fact using the formalism developed above.
Proposition \ref{prp:be} immediately yields the convergence properties stated in \eqref{eq:fokkerdecay}
as well as the regularization estimate \eqref{eq:auxiliary}.

We proceed to analyze \eqref{eq:dlss}.
To begin with, let us rewrite --- by integration by parts --- the Fisher information $\fish$ from \eqref{eq:entfish} 
with the help of $\Delta_\pi$ introduced in \eqref{eq:KDH}:
\begin{align}
  \label{eq:fishnew}
  \fish(w) = - \intcube \log(w/\pi) \,\nabla\cdot(w\nabla(w/\pi))\dd x
  = - \intcube \log(w/\pi)\,\Delta_\pi w \dd x.
\end{align}
From this representation, it is immediate to deduce the relation \eqref{eq:miracle} between entropy and Fisher information, 
i.e., that
\begin{align*}
  |\partial\ent|^2(w) = - \dff_w\ent[\Delta_\pi w] = \fish(w).
\end{align*}
Next, we use \eqref{eq:fishnew} to compute the first variation of $\fish$:
\begin{align}
  \label{eq:prediscrete}
  \dff_u\fish[\xi] 
  = -\intcube \big[(\xi/u)\,\Delta_\pi u + \log(u/\pi)\,\Delta_\pi\xi\big]\dd x
  = -\intcube \left[ \frac{\Delta_\pi u}{u} + \Delta_\pi^\star\log(u/\pi) \right]\xi\dd x,
\end{align}
where $\Delta_\pi^\star$ is the $L^2({\rm d}x)$-adjoint of $\Delta_\pi$,
that is
\begin{align*}
  \Delta_\pi^\star\log(u/\pi)
  = \Delta_\pi^\star\log u + \Delta_\pi^\star V
  = \Delta\log u - \frac{\nabla V\cdot\nabla u}u + \Delta V - |\nabla V|^2.
\end{align*}
We thus obtain
\begin{align*}
  \dff_u\fish[\xi] 
  = - \intcube \left[ \frac{\Delta u}{u} + \Delta\log u + 2\Delta V-|\nabla V|^2 \right] 
  \xi\dd x.
\end{align*}
From this and the relation \eqref{eq:VW} between $V$ and $W$,
it is obvious that \eqref{eq:dlss} can be written as the gradient flow of $\fish$:
\begin{align*}
  \partial_tu_t = F_\fish(u_t) \qtextq{with} 
  F_\fish(u) = - \ons_u\dff_u\fish = - \nabla\cdot\left(u\nabla \left[ \frac{\Delta u}{u} + \Delta\log u - W \right]\right).
\end{align*}
\begin{rmk}
  For later reference, we point out that in view of \eqref{eq:prediscrete},
  the equation \eqref{eq:dlss} can be equivalently written in the form
  \begin{align}
    \label{eq:dlss2}
    \partial_t u = \ons_u\left(\frac{\Delta_\pi u}{u} + \Delta_\pi^\star\log(u/\pi)\right).
  \end{align}
  This is the representation which naturally appears after discretization, see \eqref{eq:ddlss2} below.
\end{rmk}
In combination, this means that Proposition \ref{prp:mms} applies to solutions $u_t$ of \eqref{eq:dlss}.
The respective estimates \eqref{eq:entddown} and \eqref{eq:fishddown} turn into \eqref{eq:dlssdecay},
and \eqref{eq:fishbyent2} becomes \eqref{eq:auxiliary}.


\section{Discretization of the Fokker-Planck equation}
\label{sct:dfokker}

\subsection{Finite volume discretization}
For given $N\in\setN$, define the length parameter $h:=1/N$,
and introduce the $d$-dimensional cubic lattice of side length $N$,
\[ J^h:=\{1,\ldots,N\}^d\subset\mathbb{Z}^d.\]
Multi-indices in $J^h$ are denoted by $\bi$ and $\bj$, 
and we write $\nbh\bi\bj$ if $\bi$ and $\bj$ are neighbors, i.e., $|\bi-\bj|=1$.
Intuitively, each $\bj\in J^h$ labels a subcube
\[ \omega_\bj := \big(h(j_1-1),hj_1\big)\times\cdots\times\big(h(j_d-1),hj_d\big)\subset\cube\]
of side length $h$ in $\cube$,
and each vector $U\in\setR^{J_h}$ is associated to a function $u^h\in L^\infty(\cube)$
that is piecewise constant on each $\omega_\bj$:
\[ u^h(x) = U_\bj \qtext{for all $x\in\omega_\bj$.} \]
In this spirit, we refer to
\[\prb(J^h) := \left\{U\in\setR_+^{J^h}\,;\,h^d\sum_{\bj\in J^h}U_\bj=1\right\}\]
as the space of positive probability densities on $J^h$;
indeed, for each $U\in\prb(J^h)$,
\[ \intcube u^h(x)\dd x = h^d\sum_{\bj\in J^h}U_\bj = 1. \]
Both vectors $\Xi\in\tg_U\prb(J^h)$ and cotangent vectors $P\in\tg_U^\star\prb(J^h)$
are identified with elements in $\setR^{J^h}$ of vanishing mean,
\[ h^d\sum_{\bi}P_\bi = 0, \quad h^d\sum_{\bj}\Xi_\bj = 0,\]
and their pairing is given by
\[ P[\Xi] = h^d\sum_{\bj\in J^h}P_\bj\Xi_\bj.\]
Next, we introduce a discrete approximation $\Pi^h\in\prb(J^h)$ of the steady state $\pi$ from \eqref{eq:steady}.
First, define vectors $V^{[1],h},\ldots,V^{[d],h}\in\setR^N$ by
\begin{align}
  \label{eq:dVk}
  V^{[k],h}_j = \frac1h\int_{h(j-1)}^{hj}V^{[k]}(r)\dd r,
\end{align}
and accordingly $\Pi^{[1],h},\ldots,\Pi^{[d],h}\in\prb(\{1,\ldots,N\})$ by
\begin{align}
  \label{eq:dPik}
  \Pi^{[k],h}_j = \frac1{Z^{[k],h}}\exp\big(-V^{[k],h}_j\big),
\end{align}
with the appropriate choice of the normalization constant $Z^{[k],h}>0$.
Now, $\Pi^h\in\prb(J^h)$ itself is defined such that it inherits the product structure \eqref{eq:factorize}:
\begin{align}
  \label{eq:dfactor}
  \Pi^h_\bj = \Pi_{j_1}^{[1],h}\cdots\Pi_{j_d}^{[d],h}.
\end{align}
Since $V$ is smooth, the respective piecewise constant densities $\pi^h$ 
converge to $\pi$ uniformly on $\cube$ as $h\downarrow0$.
\begin{lem}
  \label{lem:dpi}
  The piecewise constant representation $\pi^h\in\prb(\cube)$ with respective values $\Pi^h_\bj$ on the cubes $\omega_\bj$
  is the unique minimizer of $\ent$ on the subspace of piecewise constant densities in $\prb(\cube)$.
  Moreover,
  \begin{align}
    \label{eq:gammah}
    \gamma^h := \ent(\pi^h) =  \log\frac{Z^{[1]}\cdots Z^{[d]}}{Z^{[1],h}\cdots Z^{[d],h}}.
  \end{align}
\end{lem}
This lemma justifies the definition of the discretized entropy $\ent^h$ in \eqref{eq:dmiracle}.
\begin{rmk}
  It is easily seen that $Z^{[k],h}\searrow Z^{[k]}$ for each $k=1,\ldots,d$ as $h\searrow0$.
  Hence $\gamma^h\searrow0$.
\end{rmk}
\begin{proof}
  If $u^h$ is piecewise constant on the boxes $\omega_\bj$ with respective values $U_\bj$, 
  then
  \begin{align*}
    \ent(u^h) = \intcube u^h\log(u^h/\pi)\dd x
    &= \sum_{\bj\in J^h} \int_{\omega_\bj} U_\bj\big(\log U_\bj-\log\pi\big)\dd x \\
    &= h^d\sum_{\bj\in J^h} U_\bj\left(\log U_\bj +\log Z + \frac1{h^d}\int_{\omega_\bj}V(x)\dd x\right) \\
    &= h^d\sum_{\bj\in J^h} U_\bj\log\big(U_\bj/\Pi^h_\bj\big) + \log\frac{Z^{[1]}\cdots Z^{[d]}}{Z^{[1],h}\cdots Z^{[d],h}}.
  \end{align*}
  For the last line, we have used the property \eqref{eq:decompose} of $V$, 
  which yields that
  \begin{align*}
    \frac1{h^d}\int_{\omega_\bj}V(x)\dd x
    = \frac1{h^d}\int_{\omega_\bj} \big(V^{[1]}(x_1)+\cdots+V^{[d]}(x_d)\big)\dd x
    = V^{[1],h}_{j_1} + \cdots + V^{[d],h}_{j_d},
  \end{align*}
  the property $Z=Z^{[1]}\cdots Z^{[d]}$,
  and the definition of $\Pi^h$ in \eqref{eq:dPik}\&\eqref{eq:dfactor} above.
  Since both $U^h,\Pi^h\in\prb(J^h)$, we may further write
  \begin{align*}
    \ent(u^h) = h^d\sum_{\bj\in J^h} \Pi^h_\bj\Big(1+\big(U_\bj/\Pi^h_\bj\big)\big[\log\big(U_\bj/\Pi^h_\bj\big)-1\big]\Big) 
    + \log\frac{Z^{[1]}\cdots Z^{[d]}}{Z^{[1],h}\cdots Z^{[d],h}}.
  \end{align*}
  Using that $r\mapsto r(\log r-1)+1$ is strictly convex with minimum zero attained at $r=1$, 
  Jensen's inequality implies that
  \begin{align*}
    \ent(u^h) \ge \log\frac{Z^{[1]}\cdots Z^{[d]}}{Z^{[1],h}\cdots Z^{[d],h}},
  \end{align*}
  with equality if and only if $u^h=\pi^h$.
\end{proof}

\subsection{Discretized Fokker-Planck equation}
We implicitly introduce a metric on $\prb(J^h)$ by means of 
the Onsager operator $\ons^h_U:\tg_U^\star\prb(J^h)\to\tg_U\prb(J^h)$ with
\begin{align}
  \label{eq:dons}
  P\big[\ons^h_UQ\big]
  = {h^d}\sum_{\nbh\bi\bj} \sqrt{\Pi^h_\bi\Pi^h_\bj}\Lambda_{\bi\bj}(U)\,
  \left(\frac{P_\bi-P_\bj}h\right) \left(\frac{Q_\bi-Q_\bj}h\right),
\end{align}
for all $P,Q\in\tg_U^\star\prb(J^h)$,
where the sum runs over all pairs of neighboring indices $\nbh{\bi}{\bj}$, i.e., over all edges of unit length in $J^h$,
and $\Lambda_{\bi\bj}(U)$ is an abbreviation of 
\begin{align*}
  \Lambda_{\bi\bj}(U) = \Lambda\left(\frac{U_\bi}{\Pi^h_\bi},\frac{U_\bj}{\Pi^h_\bj}\right),
\end{align*}
with the logarithmic mean $\Lambda:\setR_+\times\setR_+\to\setR_+$,
given by
\begin{align}
  \label{eq:logmean}
  \Lambda(a,b) 
  = \int_0^1 a^{1-x} b^x \,{\rm d}x = \begin{cases}\frac{a-b}{\log a-\log b} & \text{if $a\neq b$}, \\ a & \text{if $a=b$}.\end{cases}
\end{align}
It has been shown in \cite{Maas} that $\ons^h$ induces a distance on $\prb(J^h)$,
which extends to the closure $\prbnn(J^h)$ of merely non-negative probability densities.
The resulting metric space is geodesic and complete.
\begin{rmk}
  The definition \eqref{eq:dons} of the discrete Onsager operator above is consistent
  with that of the Onsager operator for the $L^2$-Wasserstein metric on $\prb(\cube)$ from \eqref{eq:ons}.
  To see this relation,
  let a smooth and positive density $u\in\prb(\cube)$ and two smooth functions $p,q:\cube\to\setR$ be given.
  For $h\downarrow0$, let $U^h\in\prb(J^h)$ and $P^h,Q^h\in\tg_{U^h}^\star\prb(J^h)$ be approximations 
  of $u$ and $p,q$ in the sense that their piecewise constant interpolations $u^h,p^h,q^h\in L^\infty(\Omega)$
  converge to the respective $u,p,q$ uniformly.
  Further, for each $\bj\in J^h$, we introduce the center $x^h_\bj$ of the $\bj$th cube,
  \begin{align*}
    x^h_\bj := \big(h (j_1-1/2),\ldots,h(j_d-1/2)\big).
  \end{align*}
  Since the values $U^h_\bi$ and $U^h_\bj$ at neighboring sites $\nbh\bi\bj$ are $O(h)$-close to each other,
  the logarithmic, geometric and arithmetic mean of $U_\bi/\Pi^h_\bi$ and $U_\bj/\Pi^h_\bj$ are $O(h)$-close to each other as well.
  Hence, we have
  \begin{align*}
    \sqrt{\Pi^h_\bi\Pi^h_\bj}\Lambda_{\bi\bj}(U^h) 
    = \sqrt{\Pi^h_\bi\Pi^h_\bj}\sqrt{\frac{U^h_\bi}{\Pi^h_\bi}\,\frac{U^h_\bj}{\Pi^h_\bj}} + O(h) 
    = \sqrt{U^h_\bi U^h_\bj}+O(h)
    = \frac12\big(u(x^h_\bi)+u(x^h_\bj)\big) + O(h)
  \end{align*}
  inside the definition \eqref{eq:dons}.
  Further, due to the square grid combinatorics of $J^h$,
  \begin{align*}
    \frac{P^h_\bi-P^h_\bj}h = (\bi-\bj)\cdot\nabla p\left(\frac{x^h_\bi+x^h_\bj}2\right) + O(h),
  \end{align*}
  and similarly for the difference quotients of $Q^h$.
  Working out the combinatorics, one obtains from the definition of the discretize Onsager operator in \eqref{eq:dons} 
  the following integral approximation:
  \begin{align*}
   & P\big[\ons^h_UQ\big]
  \\  &= \frac{h^d}2\sum_{\bj\in J^h} \left[\big(u(x^h_\bj)+O(h)\big)
      \sum_{k=1}^d\big(\ee_{x_k}\cdot\nabla p(x^h_\bj+h/2\ee_k)+O(h)\big)\,\big(\ee_k\cdot\nabla q(x^h_\bj+h/2\ee_k)+O(h)\big)\right]\\
    &= \intcube u(x)\,\nabla p(x)\cdot\nabla q(x)\dd x + O(h).
  \end{align*}
  The last expression is an approximation of the original Onsager operator from \eqref{eq:ons}.
\end{rmk}
As announced in \eqref{eq:dmiracle}, the entropy functional $\ent^h$ on $\prb(J^h)$ is
defined by restriction of the original entropy $\ent$,
\begin{align*}
  \ent^h(U) = h^d\sum_{\bj\in J_N} U_\bj\log(U_\bj/\Pi^h_\bj) -\gamma^h = \ent(u^h)-\gamma^h,
\end{align*}
where $\gamma^h$ defined in \eqref{eq:gammah} is such that 
the convex functional $\ent^h(U)$ is non-negative for all $U\in\prb(J^h)$,
and vanishes precisely for $U=\Pi^h$ given in \eqref{eq:dfactor}.
Accordingly, introduce the discretization of the Fokker-Planck operator $\Delta_\pi$ on $\prb(J^h)$ by
\begin{align}
  \label{eq:dKDH}
  \mc^hU := - \ons_U^h\dff_U\ent^h.
\end{align}
The representation as a linear operator is justified by the following.
\begin{lem}
  The discrete Fokker-Planck operator $\mc^h$ is linear 
  on the simplex $\prb(J^h)$:
  \begin{align*}
    (\mc^hU)_\bi = \sum_{\bj\in J^h}\mc^h_{\bi\bj}U_\bj \qtext{for all $U\in\prb(J^h)$,}
  \end{align*}
  with the matrix elements of $\mc^h\in\setR^{J^h\times J^h}$ being given by
  \begin{align}
    \label{eq:Qij}
    \mc^h_{\bi\bj} =
    \begin{cases}
      h^{-2}\sqrt{\Pi^h_\bi/\Pi^h_\bj} & \text{if $\nbh\bi\bj$}, \\ 
      - \displaystyle{h^{-2}\sum_{\bi':\nbh{\bi'}\bj}\sqrt{\Pi^h_{\bi'}/\Pi^h_\bj}} & \text{if $\bi=\bj$}, \\
      0 & \text{otherwise}.
    \end{cases}
  \end{align}
  Moreover, the adjoint operator $\mathcal{L}^h = (\mc^h)^*$ given by $(\mathcal{L}^h \psi)_\bi = \sum_{\bj\in J^h}\mc^h_{\bj\bi}\psi_\bj$
  is the generator of an irreducible and reversible Markov chain on $J^h$ 
  with invariant distribution $\Pi^h$.
\end{lem}
\begin{proof}
  First observe that, at each $U\in\prb(J^h)$,
  \begin{align*}
    \dff_U\ent^h[\Xi] = h^d\sum_{\bj\in J^h} \big(1+\log(U_\bj/\Pi^h_\bj)\big)\Xi_\bj
    \qtext{for all $\Xi\in\tg_U\prb(J^h)$.}
  \end{align*}
  Thus, by definition of $\Lambda$, we have for each $P\in\tg_U^\star\prb(J^h)$:
 \begin{align*}
    P\big[-\ons^h_U\dff_U\ent^h\big]
    &= -{h^d}\sum_{\nbh\bi\bj}\sqrt{\Pi^h_\bi\Pi^h_\bj}\Lambda_{\bi\bj}(U)
    \left(\frac{\log(U_\bi/\Pi^h_\bi)-\log(U_\bj/\Pi^h_\bj)}h\right)\left(\frac{P_\bi-P_\bj}h\right) \\
    &= -{h^d}\sum_{\nbh\bi\bj}\sqrt{\Pi^h_\bi\Pi^h_\bj}
    \left(\frac{U_\bi/\Pi^h_\bi-U_\bj/\Pi^h_\bj}h\right)\left(\frac{P_\bi-P_\bj}h\right) \\
    &=  h^{d-2}\sum_{\bi}P_\bi\left[\sum_{\bj:\nbh\bi\bj}\left(\sqrt{\frac{\Pi^h_\bi}{\Pi^h_\bj}}U_\bj-\sqrt{\frac{\Pi^h_\bj}{\Pi^h_\bi}}U_\bi\right)\right].
  \end{align*}
  This shows the linearity of the operator in \eqref{eq:dKDH},
  and yields the representation \eqref{eq:Qij}.

  $\mc^h$ being the adjoint generator of a Markov chain means that 
  all of its off-diagonal entries are non-negative, and that the column sums vanish.
  Both properties are immediately verified by inspection of \eqref{eq:Qij}.
  Irreducibility means that for any two indices $\bi_*,\bi^*$, 
  one finds a chain $(\bi_m)_{m=0,\ldots,M}$ of indices $\bi_m\in J^h$ with $\bi_0=\bi_*$ and $\bi_M=\bi^*$
  such that $\mc^h_{\bi_m\bi_{m-1}}>0$ for all $m=1,\ldots,M$.
  Since $\mc^h_{\bi\bj}>0$ whenever $\nbh\bi\bj$, 
  one may take for $(\bi_m)_{m=0,\ldots,M}$ any chain with $\nbh{\bi_{m-1}}{\bi_m}$ connecting $\bi_*$ with $\bi^*$.
  For reversibility, we need to verify the detailed balance condition
  \begin{align}
    \label{eq:balance}
    \mc^h_{\bi\bj}\Pi^h_\bj = \mc^h_{\bj\bi}\Pi^h_\bi \qtext{for all $\bi,\bj\in J^h$.}
  \end{align}
  This again is an immediate consequence of the representation \eqref{eq:Qij}.
  Note that \eqref{eq:balance} together with the Markov property implies $\mc^h\Pi^h=\Pi^h$,
  i.e., $\Pi^h$ is indeed an (in fact: the unique) invariant distribution.
\end{proof}

\subsection{$\lambda$-contractivity of the Fokker-Planck flow}
The goal of this section is to prove:
\begin{prp}
  \label{prp:dfokkerconvex}
  The flow $F_{\ent^h}$ from \eqref{eq:dKDH}
  satisfies \eqref{eq:cvx-e-d}
  with
  \begin{align}
    \label{eq:lambdah}
    \lambda^h = \frac2{h^2}\left( 1-\exp\left(-\frac{h^2}2\lambda\right)\right) = \lambda + O(h^2).
  \end{align}
\end{prp}
This result appears to be novel for dimensions $d>1$, 
but its proof is obtained by combination of two results from the literature.
The key observation is that, in view of the factorization property \eqref{eq:dfactor},
the space $\prb(J^h)$ carries a natural tensorial structure 
that is compatible with the evolution \eqref{eq:dKDH} of the spatially discrete Fokker-Planck equation.
More precisely:
\begin{lem}
  For each pair of indices $\bi,\bj\in J^h$,
  \begin{align}
    \label{eq:decompQ}
    \mc^h_{\bi\bj} 
    = \mc^{[1],h}_{i_1j_1}\delta_{i_2j_2}\delta_{i_3j_3}\cdots\delta_{i_dj_d}
    + \delta_{i_1j_1}\mc^{[2],h}_{i_2j_2}\delta_{i_3j_3}\cdots\delta_{i_dj_d}
    + \cdots + \delta_{i_1j_1}\delta_{i_2j_2}\delta_{i_3j_3}\cdots\mc^{[d],h}_{i_dj_d}.
  \end{align}
  Here each $\mc^{[k],h}\in\setR^{N\times N}$ is a tri-diagonal matrix,
  \begin{align}
    \label{eq:Qk}
    \mc^{[k],h}_{ij} =
    \begin{pmatrix}
      -\sigma^{[k],h}_1 & \beta^{[k],h}_1 & &  &  &  \\[5pt]
      \alpha^{[k],h}_1 & -\sigma^{[k],h}_2 & \beta^{[k],h}_2 & & & \\[5pt]
      & \alpha^{[k],h}_2 & -\sigma^{[k],h}_3 & \ddots & & \\[5pt]
      &  & \ddots & \ddots & \ddots & & \\[5pt]
      & & & \ddots & -\sigma^{[k],h}_{N-1} & \beta^{[k],h}_{N-1} \\[5pt]
      & & & & \alpha^{[k],h}_{N-1} & -\sigma^{[k],h}_N
    \end{pmatrix},
  \end{align}
  and the entries $\alpha^{[k],h}_j,\,\beta^{[k],h}_j,\,\sigma^{[k],h}_j$ are given by
  \begin{align}
    \label{eq:defalpha}
    \alpha^{[k],h}_j &= h^{-2}\sqrt{\Pi^{[k],h}_{j+1}/\Pi^{[k],h}_j} = h^{-2}\exp\left(\frac12(V^{[k],h}_{j}-V^{[k],h}_{j+1})\right), \\
    \label{eq:defbeta}
    \beta^{[k],h}_j &= h^{-2}\sqrt{\Pi^{[k],h}_{j-1}/\Pi^{[k],h}_j} = h^{-2}\exp\left(\frac12(V^{[k],h}_j-V^{[k],h}_{j-1})\right), \\
    \sigma^{[k],h}_j &= \alpha^{[k],h}_j+\beta^{[k],h}_{j-1} \quad(j=2,\ldots,N-1),\quad
    \sigma^{[k],h}_1 = \alpha^{[k],h}_1,\quad \sigma^{[k],h}_N = \beta^{[k],h}_{N-1}.
  \end{align}
\end{lem}
\begin{proof}
  This follows directly from the representation \eqref{eq:Qij} of $\mc^h$'s entries.  
\end{proof}
Naturally, there is an associated decomposition of the operator $\ons^h$ on $\prb(J^h)\subset\setR_+^{J^h}$ 
into a sum of operators $\ons^{[k],h}$, with each $\ons^{[k],h}$ acting on the smaller state spaces $\setR^N_+$ by
\begin{align}
  \label{eq:donsAlex}
  \tilde P[\ons^{[k],h}_{\tilde U}\tilde Q] 
  = h\sum_{j=1}^{N-1}\sqrt{\Pi^{[k],h}_j\Pi^{[k],h}_{j+1}}
  \Lambda\left(\frac{\tilde U_j}{\Pi^{[k],h}_j},\frac{\tilde U_{j+1}}{\Pi^{[k],h}_{j+1}}\right)\,
  \left(\frac{\tilde P_j-\tilde P_{j+1}}h\right)\left(\frac{\tilde Q_j-\tilde Q_{j+1}}h\right), 
\end{align}
for $\tilde U\in\setR_+^N$ and $P,Q\in\setR^N$, recalling the notation $\Pi^{[k],h}$ introduced in \eqref{eq:dfactor}.
For definiteness, select a spatial direction $k\in\{1,\ldots,d\}$ and introduce accordingly:
$J^{[k],h}\subset J^h$ as the set of indices $\bj'$ with $j_k'=0$;
for each $U\in\prb(J^h)$ and $\bj'\in J^{[k],h}$ the projection $U^{[k]}_{\bj'}\in\setR_+^N$ 
such that $(U^{[k]}_{\bj'})_j=U_{(j_1',\ldots,j_{k-1}',j,j_{k+1}',\dots,j_d')}$;
similarly, for $P,Q\in\tg_U^*\prb(J^h)$ the projections $P^{[k]}_{\bj'},Q^{[k]}_{\bj'}\in\setR^N$.
It is then easily verified that
\begin{align}
  \label{eq:decompK}
  P[\ons^h_UQ] = \sum_{k=1}^d h^{d-1}\sum_{\bj'\in J^{[k],h}} P^{[k]}_{\bj'}\big[\ons^{[k],h}_{U^{[k]}_{\bj'}}Q^{[k]}_{\bj'}\big].
\end{align}
Indeed, one only needs to take into account the square-grid structure of $J^h$, 
and the fact that for arbitrary $\nbh{\bi}{\bj}$ with $j:=j_k=i_k+1$, 
\begin{align*}
  \sqrt{\Pi^h_\bi\Pi^h_\bj}\Lambda\left(\frac{A}{\Pi^h_\bi},\frac{B}{\Pi^h_\bj}\right)
  = \sqrt{\Pi^{[k],h}_i\Pi^{[k],h}_{i+1}}\Lambda\left(\frac{A}{\Pi^{[k],h}_i},\frac{B}{\Pi^{[k],h}_{i+1}}\right)
  \qtext{for all $A,B>0$},
\end{align*}
thanks to the factorization \eqref{eq:dfactor}, and to the properties of the logarithmic mean \eqref{eq:logmean}.
\begin{lem}\label{lem:1d-cvx}
  For each $k\in\{1,\ldots,d\}$,
  the matrix $\mc^{[k],h}$ induces a $\lambda^h$-contractive flow on $\setR^N$
  with respect to the corresponding Onsager operator $\ons^{h,[k]}$.
\end{lem}
\begin{proof}
  Eventually, we will apply \cite[Theorem 3.1]{CDPP09}, which deals precisely 
  with matrices $\mc^{h,[k]}$ and operators $\ons^{[k],h}$ of the forms \eqref{eq:Qk} and \eqref{eq:donsAlex}, respectively.
  But first, we establish the following auxiliary estimate
  \begin{align}
    \label{eq:rootest}
    \sqrt{\Pi^{[k],h}_{i+1}\Pi^{[k],h}_{i-1}} \le \left(1-\frac{h^2}2\lambda^h\right)\Pi^{[k],h}_i.
  \end{align}
  Indeed, by $\lambda$-convexity of $V^{[k]}$, we have that
  \begin{align*}
    \frac12\big(V^{[k]}(x_k+h)+V^{[k]}(x_k-h)\big) \ge V^{[k]}(x_k) + \frac\lambda2h^2.
  \end{align*}
  Integration of this inequality from $x_k=(i-1)h$ to $x_k=ih$ yields
  \begin{align*}
    \frac12\big(V^{[k],h}_{i+1}+V^{[k],h}_{i-1}\big) \ge V^{[k],h}_i + \frac\lambda2h^2,
  \end{align*}
  which further implies that
  \begin{align*}
    \sqrt{\exp\big(-V^{[k],h}_{i+1})\exp(-V^{[k],h}_{i-1})} \le \exp\left(-\frac{h^2}2\lambda\right)\exp(-V^{[k],h}_i).
  \end{align*}
  Recalling the definition \eqref{eq:dPik} of $\Pi^{[k],h}$,
  and the definition \eqref{eq:lambdah} of $\lambda^h$,
  the estimate \eqref{eq:rootest} follows.

  An immediate consequence of \eqref{eq:rootest} is the validy of the monotonicity hypotheses
  \begin{align}
    \label{eq:abmonotone}
    \alpha^{[k]}_j\le\alpha^{[k]}_{j-1}
    \qtextq{and}
    \beta^{[k]}_j\le\beta^{[k]}_{j+1}.
  \end{align}
  Therefore, \cite[Theorem 3.1]{CDPP09} is applicable.
  It provides the \eqref{eq:cvx-e-d} for $\mc^{[k],h}$ with respect to the Onsager operator $\ons^{[k],h}$,
  for each
  \begin{align}\label{eq:cdpp-ineq}
    \lambda \leq  \lambda^{*} := 
    \min_{i=2,\ldots,N-1}{(\alpha^{[k]}_i-\alpha^{[k]}_{i+1}) +  (\beta^{[k]}_i-\beta^{[k]}_{i-1})}.
  \end{align}
  Now, from the definitions \eqref{eq:defalpha} and \eqref{eq:defbeta} of $\alpha^{[k]}$ and $\beta^{[k]}$,
  it follows via \eqref{eq:rootest} that
  \begin{equation}\begin{aligned}\label{eq:el}
      \alpha^{[k],h}_i-\alpha^{[k],h}_{i+1}
      &= h^{-2}\sqrt{\Pi^{[k],h}_{i+1}/\Pi^{[k],h}_i}-h^{-2}\sqrt{\Pi^{[k],h}_{i+2}/\Pi^{[k],h}_{i+1}}
      \ge h^{-2}\left(\frac{h^2}2\lambda^h\right) \sqrt{\Pi^{[k],h}_{i+1}/\Pi^{[k],h}_i} \\
      \beta^{[k],h}_i-\beta^{[k],h}_{i-1}
      &= h^{-2}\sqrt{\Pi^{[k],h}_i/\Pi^{[k],h}_{i+1}}-h^{-2}\sqrt{\Pi^{[k],h}_{i-1}/\Pi^{[k],h}_i}
      \ge h^{-2}\left(\frac{h^2}2\lambda^h\right) \sqrt{\Pi^{[k],h}_i/\Pi^{[k],h}_{i+1}},
    \end{aligned}\end{equation}
  which implies that 
  \begin{align}\label{eq:constant}
    \lambda^{*} \geq \lambda^h \min_{1=2, \ldots, N-1} \cosh(V^{[k],h}_i - V^{[k],h}_{i+1}) \geq \lambda^h,
  \end{align}
  as desired.
\end{proof}
Proposition \ref{prp:dfokkerconvex} follows immediately by combining Lemma \ref{lem:1d-cvx} with the tensorisation result from Theorem \ref{thm:tensorisation} below.
\begin{rmk}
  \label{rmk:MielkevsCaputo}
  In the setting of Proposition \ref{lem:1d-cvx}, 
  it is possible to prove the stronger property of $\lambda$-convexity (with a slightly worse constant) with a minor modification of the proof. 
  Instead of using \cite[Theorem 3.1]{CDPP09} to obtain  the inequality \eqref{eq:cdpp-ineq} as above, 
  one could apply Mielke's criterion from \cite[Theorem 5.1]{Mielke} to obtain $\tilde\lambda^h$-convexity,
  with
  \begin{align}
    \label{eq:lambdastar}
    \tilde\lambda^h := 2\min_{i=2,\ldots,N-1}\sqrt{(\alpha^{[k]}_i-\alpha^{[k]}_{i+1}) (\beta^{[k]}_i-\beta^{[k]}_{i-1})}.
  \end{align}
  Note that the arithmetic mean in \eqref{eq:cdpp-ineq} is replaced by a geometric mean in \eqref{eq:lambdastar}. 
  It is easily seen that $\lambda^h>\tilde\lambda^h$, 
  but that the difference $\tilde\lambda^h-\lambda^h=O(h^2)$ becomes negligible in the discrete-to-continuous limit $h\downarrow0$.
  In view of the tensorisation result from \eqref{eq:tensor-cvx} the result remains valid in any dimension with the same constant.
\end{rmk}

\subsection{Tensorisation of convex entropy decay for Markov chains}
\label{sct:tensor}
\newcommand{\calI}{\mathcal{I}}
\newcommand{\calL}{\mathcal{L}}
\newcommand{\calG}{\mathcal{G}}
In this section, we sketch the proof for stability of the inequality \eqref{eq:cvx-e-d} under tensorization.
This result is independent of the discretization and might be of interest on its own right.

We need to fix some notations.
First, we recall an alternative representation of a continuous time Markov chain on a finite set $\calI$: 
the generator $\calL : L^\infty(\calI) \to L^\infty(\calI)$ can be written as
\begin{align*}
  (\calL \psi)_i = \sum_{\delta \in \calG} c_{i,\delta} \big(\psi_{\delta(i)} - \psi_i\big)\;.
\end{align*}
Here, $\calG$ is a set of maps from $\calI$ to $\calI$ representing the possible jumps, 
and $c_{i,\delta} \geq 0$ denotes the jump rate from $i$ to $\delta(i)$. 
For brevity, we shall write $\nabla_\delta \psi_i := \psi_{\delta(i)} - \psi_i$. 

Throughout this section we assume that the following reversibility conditions are satisfied:
\begin{itemize}
\item for every $\delta \in \calG$ there exists a unique $\delta^{-1} \in \calG$ 
  satisfying $\delta^{-1}(\delta(i)) = i$ for all $i$ with $c_{i,\delta} > 0$;
\item there exists a probability measure $\pi=(\pi_i)_{i\in\calI}$ on $\calI$ such that $\pi_i>0$ for all $i$, 
  and 
  \begin{align*}
    \sum_{i \in \calI, \delta \in \calG} F(i,\delta) c_{i,\delta} \pi_i 
    = \sum_{i \in \calI, \delta \in \calG} F(\delta(i),\delta^{-1}) c_{i,\delta} \pi_i 
  \end{align*}
  for all $F : \calI \times \calG \to \setR$.
\end{itemize}
The relative entropy functional $\ent_\pi : \prb(\calI) \to \setR$ is given by
\begin{align*}
  \ent_\pi(u) = \sum_{i \in \calI} u_i \log(u_i/\pi_i)\;.
\end{align*}
In accordance with the situation described before,
we introduce an Onsager operator $\ons$ such that $\calL^*(u) = \ons_u\dff_u\ent_\pi(u)$:
\begin{align*}
  p[\ons_u q] = \frac12 \sum_{i \in \calI, \delta \in \calG}
  \pi_i c_{i,\delta} \Lambda_{i,\delta}(u) \nabla_\delta p_i  \nabla_\delta q_i \;,
  \quad\quad\text{where}\quad
  \Lambda_{i,\delta}(u) = \Lambda \bigg(\frac{u_i}{\pi_i}, \frac{u_{\delta(i)}}{\pi_{\delta(i)}} \bigg).
\end{align*}
For later reference, let us calculate the Hessian $\cvx$:
it follows from the definition that
\begin{align*}
  \cvx_u[p,p] 
  & = -\frac12\sum\limits_{i\in\calI}\sum\limits_{\delta,\eta\in \calG}
  \pi_i c_{i,\delta}  \Lambda_{i,\delta}(u) 
  \nabla_\delta p_i   \Big[ c_{\delta(i),\eta} \nabla_\eta p_{\delta(i)} - c_{i,\eta} \nabla_\eta p_i \Big]\\
  & \qquad + \frac{1}{4}\sum\limits_{i\in\calI}\sum\limits_{\delta,\eta\in \calG}
  \pi_i c_{i,\delta} 
  \big(\nabla_\delta p_i\big)^2\Big[ c_{i,\eta} \Lambda^1_{i,\delta}(u)
  \nabla_\eta(u/\pi)_i +
  c_{\delta(i),\eta} \Lambda^2_{i,\delta}(u) \nabla_\eta(u/\pi)_{\delta(i)}\Big] \\
  &  =: \sum_{i \in \calI} \sum_{\delta, \eta \in \calG} F_{u,p}(i,\delta,\eta),
\end{align*}
where 
$\Lambda_{i,\delta}^m(u) = \partial_m\Lambda \big(\frac{u_i}{\pi_i}, \frac{u_{\delta(i)}}{\pi_{\delta(i)}} \big)$ for $m = 1,2$.
See \cite{EM,Mielke} for details.

Now consider a collection of Markov chains $(\calI^{k}, \calL^{k}, \pi^{k})$ for ${k} = 1, \ldots, N$.
The corresponding product chain $(\calI^\otimes, \calL^\otimes, \pi^\otimes)$ is defined by 
\begin{align*}
  \calI^\otimes & = \calI^1 \times \cdots \times \calI^N, \qquad
  (\calL^\otimes \psi)_\bi  = \sum_{k} \calL^{k} \psi_{\bi} \quad 
  \text{ for } \psi \in L^\infty(\calI),\qquad
  \pi^\otimes  = \pi^1 \otimes \cdots \otimes \pi^N.
\end{align*}
Here it is understood that $\calL^{k}$ acts on the ${k}$-th coordinate of $\psi$. 
The corresponding Onsager operator and Hessian will be denoted by $\ons^\otimes$ and by $\cvx^\otimes$, respectively.
To simplify notations we shall write $\ent^{k} := \ent_{\pi^{k}}$.

It has been shown in \cite{EM} that geodesic $\lambda$-convexity is preserved under tensorisation:
\begin{align}
  \label{eq:tensor-cvx}
  \Big[ \cvx^{k} \geq \lambda_{k} \ons^{k} \quad\text{  for all }{k} = 1, \ldots, N\Big]
  \qquad   \Longrightarrow \qquad
  \Big[ \cvx^\otimes \geq \big(\min_{k} \lambda_{k}\big) \ons^\otimes \Big].
\end{align}
This result is dimension independent, i.e., the bound does not depend on $N$.
The goal of this section is to verify that the corresponding tensorisation property also holds 
for the convex entropy decay inequality \eqref{eq:cvx-e-d}. 
\begin{thm}[Tensorisation of convex entropy decay]\label{thm:tensorisation}
  Suppose that the inequality \eqref{eq:cvx-e-d} holds for each ${k} = 1, \ldots, N$:
  \begin{align*} 
    \cvx^{k}_u(\dff_u \ent^{k},\dff_u \ent^{k}) \ge \lambda_{k} \dff_u \ent^{k}[\ons_u^{k}\dff_u\ent^{k}]
  \end{align*}
  for all $u \in \prb(\calI^{k})$ and some $\lambda_{k} \in \setR$. 
  Then \eqref{eq:cvx-e-d} also holds for the product chain:
  \begin{align*} 
    \cvx^\otimes_u(\dff_u \ent_\pi,\dff_u \ent_\pi) \ge \big(\min_{k} \lambda_{k}\big) \dff_u \ent_\pi[\ons_u^\otimes\dff_u\ent_\pi]
  \end{align*}
  for all $u \in \prb(\calI^\otimes)$.
\end{thm}
The proof follows along the lines of the proof of \eqref{eq:tensor-cvx} in \cite{EM}. 
For the convenience of the reader we provide some details.
\begin{proof}
  For each $k=1,\ldots,K$, set $\calI_{\check{k}}=\prod_{{\ell} \neq {k}}\calI_{\ell}$,
  and for $\bi\in\calI^\otimes$, let $i_{\check{k}} \in \calI_{\check{k}}$ be the multi-index with the $k$th entry $i_k$ omitted.
  For a function $p : \calI^\otimes \to \setR$, 
  define its reduction $p^{i_{\check{k}} } : \calI_{k}\to\setR$ where all indices except $i_{k}$ are fixed to $i_{\check{k}}$, 
  i.e., $p^{i_{\check{k}} }_{i_{k}} = p_\bi$.
  Likewise, introduce $u^{i_{\check k}}:\calI_k\to\setR_+$.
  Finally, set $\pi^{\check {k}}=\bigotimes_{{\ell}\neq{k}}\pi^{\ell}$.

  It follows from the definitions that the Onsager matrix for the product system admits a decomposition of the form
  \begin{align}
    \label{eq:A-decompose}
    p[\ons^\otimes_u p] = \sum_{{k}=1}^N \sum\limits_{i_{\check {k}}\in\calI_{\check {k}}}  
    \pi^{\check {k}}_{i_{\check {k}}}
    p^{i_{\check {k}}}[\ons^{k}_{u^{i_{\check {k}}}} p^{i_{\check {k}}}]\;.
  \end{align}
  Also $\cvx_u^\otimes$ can be split into terms corresponding to the different components:
  \begin{align*}
    \cvx^\otimes_u[p,p] 
    = \sum_{{k}, {\ell} =1}^N \cvx^{{k},{\ell}}_u[p,p],
    \quad\text{where}\quad 
    \cvx^{{k},{\ell}}_u[p,p] 
    = \sum_{\bi \in \calI^\otimes} \sum_{\delta \in \calG^{k}, \eta \in \calG^{\ell}} F_{u,p}(\bi,\delta,\eta),
  \end{align*}
  and $\calG^{k}$ denotes the set of maps associated with the operator $\calL^{k}$. 
  It has been shown in \cite{EM} that the off-diagonal terms are non-negative (regardless of the $\lambda$-convexity properties of the components):
  \begin{align}\label{eq:off-sign}
    \cvx^{{k},{\ell}}_u[p,p] \geq 0 \quad \text{if ${k} \neq {\ell}$}
  \end{align}
  for all $u \in \prb(\calI^\otimes)$ and $p : \calI^\otimes \to \setR$.
  The on-diagonal terms satisfy
  \begin{align}\label{eq:on-diag}
    \cvx^{{k},{k}}_u[p,p] 
    =  \sum_{i_{\check {k}}\in\calI_{\check {k}}}   \pi^{\check {k}}_{i_{\check {k}}}
    \cvx^{{k}}_{u^{i_{\check {k}}}}[p^{i_{\check {k}}},p^{i_{\check {k}}}].
  \end{align}
  Now fix $u \in \prb(\calI^\otimes)$.
  Then
  \begin{align*}
    \big(\dff_u \ent_\pi\big)^{i_{\check {k}}}_{i_{k}} = \log ({u_\bi}/{\pi_\bi}) = \log (u^{i_{\check {k}}}_{i_{k}}/{\pi^{{k}}_{i_{k}}}) - \log(\pi^{\check {k}}_{i_{\check{k}}}),
    \quad
    \big(\dff_{u^{i_{\check {k}}}} \ent^k\big)_{i_k} = \log (u^{i_{\check {k}}}_{i_{k}}/{\pi^{{k}}_{i_{k}}}),
  \end{align*}
  which allows to conclude that
  \begin{align*}
    (\dff_u \ent_\pi)^{i_{\check{k}}} = \dff_{u^{i_{\check {k}}}} \ent^{{k}} - \log(\pi^{\check {k}}_{i_{\check{k}}}).
  \end{align*}
  Hence both derivatives coincide up to a constant, whose value is irrelevant, since $\cvx_u[p,p]$ depends on $p$ only through the values of its discrete derivatives $\nabla_\delta p$.
  Putting everything together, one obtains
  \begin{align*}
     \cvx^\otimes_u[\dff_u \ent_\pi, \dff_u \ent_\pi]
     &\genfrac{}{}{0pt}{1}{\text{\eqref{eq:off-sign}--\eqref{eq:on-diag}}}{\ge} 
     \sum_{{k} =1}^N \sum_{i_{\check {k}}\in\calI_{\check {k}}}   \pi^{\check {k}}_{i_{\check {k}}}
     \cvx^{{k}}_{u^{i_{\check {k}}}}[\dff_{u^{i_{\check {k}}}} \ent^{{k}},\dff_{u^{i_{\check {k}}}} \ent^{{k}}] \\
     & \genfrac{}{}{0pt}{1}{\text{\eqref{eq:cvx-e-d} for $k$}}{\ge} \sum_{{k} =1}^N \lambda_{k} \sum_{i_{\check {k}}\in\calI_{\check {k}}}   \pi^{\check {k}}_{i_{\check {k}}}
     \dff_{u^{i_{\check {k}}}}\ent^{k}[\ons^{k}_{u^{i_{\check {k}}}} \dff_{u^{i_{\check {k}}}}\ent^{k}]  
     \genfrac{}{}{0pt}{1}{\text{\eqref{eq:A-decompose}}}{\ge} \Big(\min_{k} \lambda_{k}    \Big)  \dff_u \ent_\pi[\ons^\otimes_u \dff_u \ent_\pi],
   \end{align*}
   which is the desired result.
\end{proof}

\section{Discretization of the QDD equation}
\label{sct:ddlss}
In this section, we study the gradient flow of the 
discretized Fisher information $\fish^h$, which is defined by
\begin{align}
  \label{eq:dmiracle2}
  \fish^h(U) = \dff_U\ent^h[-\mc^hU]
  = {h^d}\sum_{\nbh\bi\bj} \sqrt{\Pi^h_\bi\Pi^h_\bj}\Lambda_{\bi\bj}(U)
  \,\left(\frac{\log(U_\bi/\Pi^h_\bi)-\log(U_\bj/\Pi^h_\bj)}h\right)^2.
\end{align}
Notice that this definition is in accordance with \eqref{eq:dmiracle}.

\subsection{Existence of the gradient flow}
\begin{lem}
  \label{lem:fishy}
  $\fish^h$ is well-defined and non-negative on $\prb(J^h)$,
  with $\fish^h(U)=0$ if and only if $U=\Pi^h$.
  Moreover, $\fish^h$ has the alternative representation
  \begin{align}
    \label{eq:daltfish}
    \fish^h(U) = h^{d-2}\sum_{\nbh\bi\bj} \sqrt{\Pi^h_\bi\Pi^h_\bj}
    \left(\frac{U_\bi}{\Pi^h_\bi}-\frac{U_\bj}{\Pi^h_\bj}\right)
    \left(\log\frac{U_\bi}{\Pi^h_\bi}-\log\frac{U_\bj}{\Pi^h_\bj}\right)
  \end{align}
  for each $U\in\prb(J^h)$.
  Finally, all sublevel sets of $\fish^h$ are relatively compact in $\prb(J^h)$.
\end{lem}
\begin{rmk}
  \label{rmk:simplex}
  Since the closure of $\prb(J^h)$ in $\setR^{J^h}$ is the compact simplex
  \begin{align*}
    \prbnn(J^h) = \left\{ U\in\setR_+^{J^h}\,;\,\sum_{\bj\in J^h}U_\bj=1\right\},
  \end{align*}
  a subset of $\prb(J^h)$ is relatively compact in $\prb(J^h)$ 
  if and only if it is a closed subset of $\prbnn(J^h)$.
  A consequence is that if $A\subset\prb(J^h)$ is relatively compact, 
  then it has a positive distance $\delta>0$ to the boundary of $\prbnn(J^h)$,
  i.e.,
  \begin{align}
    \label{eq:compact}
    \inf_{U\in A}\min_{\bj\in J^h}U_\bj > 0.
  \end{align}
\end{rmk}
\begin{proof}
  Well-definedness and non-negativity are obvious from \eqref{eq:dmiracle2}.
  Since $\Lambda_{\bi\bj}(U)>0$ for each $U\in\prb(J^h)$,
  and since any two indices $\bi,\bj\in J^h$ can be connected by a sequence of neighbors,
  $\fish^h(U)=0$ holds if and only if $\log(U_\bi/\Pi^h_\bi)$ is a constant independent of $\bi$.
  That is, $U=\alpha\Pi^h$ for a global constant $\alpha>0$.
  Now $U\in\prb(J^h)$ implies $\alpha=1$, i.e., $U=\Pi^h$.

  The representation \eqref{eq:daltfish} follows immediately 
  from the definition \eqref{eq:logmean} of the logarithmic mean, 
  since $a-b = \Lambda(a,b)(\log a-\log b)$.

  It remains to prove the compactness of sublevel sets.
  By continuity of $\fish^h$, any sublevel set $A:=(\fish^h)^{-1}([0,a])$ is relatively closed in $\prb(J^h)$.
  In view of Remark \ref{rmk:simplex} above, it remains to be verified that $A$ is also closed in $\prbnn(J^h)$,
  i.e., that the closure of $A$ in $\setR^{J^h}$ does not intersect $\prbnn(J^h)\setminus\prb(J^h)$.
  Towards a contradiction, 
  assume that a sequence $(U^n)_{n\in\setN}$ in $A$ is such that $U^n\to U^*\notin\prb(J^h)$;
  we are going to show that $\fish^h(U^n)\to\infty$.
  By compactness of $\prbnn(J^h)$ in $\setR^{J^h}$, the limit $U^*$ lies in the boundary $\prbnn(J^h)\setminus\prb(J^h)$. 
  Thus, there is some $\bi_*\in J^h$ with $U^*_{\bi_*}=0$.
  On the other hand, $U^*\in\prbnn(J^h)$ implies that there is some $\bi^*\in J^h$ with $U^*_{\bi^*}>0$.
  Since $\bi_*$ and $\bi^*$ can be connected by a sequence of neighbors,
  there must exist $\bj_*,\bj^*\in J^h$ with $\nbh{\bj_*}{\bj^*}$ and $U^*_{\bj_*}=0$, $a:=U^*_{\bj^*}>0$.
  Since all the terms in the summation in \eqref{eq:daltfish} are non-negative,
  \begin{align*}
    \fish^h(U^n) &\ge h^{d-2}\sqrt{\Pi^h_\bi\Pi^h_\bj}
    \left(\frac{U^n_{\bj^*}}{\Pi^h_{\bj^*}}-\frac{U^n_{\bj_*}}{\Pi^h_{\bj_*}}\right)
    \left(\log\frac{U^n_{\bj^*}}{\Pi^h_{\bj^*}}-\log\frac{U^n_{\bj_*}}{\Pi^h_{\bj_*}}\right) \\
    &=h^{d-2}\sqrt{\Pi^h_\bi\Pi^h_\bj}
    \left[
      \underbrace{\left(\frac{U^n_{\bj^*}}{\Pi^h_{\bj^*}}-\frac{U^n_{\bj_*}}{\Pi^h_{\bj_*}}\right)\log\frac{U^n_{\bj^*}}{\Pi^h_{\bj^*}}}_{=(I)}
      +\underbrace{\left(\frac{U^n_{\bj^*}}{\Pi^h_{\bj^*}}-\frac{U^n_{\bj_*}}{\Pi^h_{\bj_*}}\right)\left(-\log\frac{U^n_{\bj_*}}{\Pi^h_{\bj_*}}\right)}_{=(II)}
    \right].
  \end{align*}
  By the choices made above,
  \begin{align*}
    (I) \to \frac{U^*_{\bj^*}}{\Pi^h_{\bj^*}}\log\frac{U^*_{\bj^*}}{\Pi^h_{\bj^*}} \ge -e^{-1},
  \end{align*}
  while for all sufficiently large $n$,
  \begin{align*}
    (II) \ge \frac12\frac{U^n_{\bj^*}}{\Pi^h_{\bj^*}}\left(-\log\frac{U^n_{\bj_*}}{\Pi^h_{\bj_*}}\right),
  \end{align*}
  which obviously diverges to $+\infty$ as $n\to\infty$.
\end{proof}
We calculated the derivative of $\fish^h$:
\begin{align*}
  \dff_U\fish^h[\Xi]
  = \dff_U^2\ent^h[-\mc^hU,\Xi] + \dff_U\ent^h[-\mc^h\Xi]
  = -h^2\sum_{\bi,\bj}\Xi_\bi\left(\frac1{U_\bi}\mc^h_{\bi\bj}U_\bj + \log(U_\bi/\Pi^h_\bi)\,\mc^h_{\bi\bj}\right) .
\end{align*}
In other words, with a certain abuse of notation, the gradient flow of $\fish^h$ is given by
\begin{align}
  \label{eq:dKDF}
  \dot U
  = -\ons^h_U\dff_U\fish^h
  = \ons^h_U\left(\frac{\mc^hU}{U} + (\mc^h)^T\log(U/\Pi^h)\right).
\end{align}
The initial value problem for this gradient flow is well-posed.
\begin{lem}
  For every initial condition $U_0\in\prb(J^h)$,
  there is a unique differentiable curve $U:\setRnn\to\prb(J^h)$ satisfying \eqref{eq:dKDF}
  with $U(0)=U_0$.
\end{lem}
\begin{proof}
  The right-hand side of \eqref{eq:dKDF} is obviously smooth in $U$.
  By the standard theory of ordinary differential equations, 
  there exists a \emph{maxial local} solution $U:[0,T)\to\prb(J^h)$.
  Here ``maximal'' means that either $T=\infty$, i.e., the local solution is global, 
  or that there is no limit point in $\prb(J^h)$ of $U(t)$ for $t\uparrow T$.
  We are now going to prove that the second alternative is impossible.

  Indeed, since $U_0\in\prb(J^h)$, we have $a:=\fish^h(U(0))<\infty$ by Lemma \ref{lem:fishy}.
  Now $U$ being a gradient flow implies that $\fish^h(U(t))\le a$ for all $t\in[0,T)$.
  That is, the curve $U$ lies in the sublevel set $A:=(\fish^h)^{-1}([0,a])$, which is compact by Lemma \ref{lem:fishy}.
  The smooth vector field $U\mapsto -\ons^h_U\dff_U\fish^h$ is bounded on $A$,
  and consequently, $U$ is uniformly Lipschitz continuous on $[0,T)$.
  Therefore, $U(t)$ has a limit in $A$ for $t\uparrow T$.
\end{proof}
\begin{rmk}
  In the obvious way, $\mc^h$ induces a linear operator $\Delta_\pi^h$ 
  on the subspace of density functions $u^h\in\prb(\cube)$ that are piecewise constant on each sub-cube $\omega_\bi$.
  In the same spirit, $\ons^h$ induces a compatible Onsager operator $\tilde\ons^h$,
  \begin{align*}
    \tilde{\ons}^h_{u^h}p^h = \ons^h_UP.
  \end{align*}
  With these notations,
  the discrete analogue \eqref{eq:dKDF} of the QDD equation \eqref{eq:dlss} can be written in the following way
  \begin{align}
    \label{eq:ddlss2}
    \partial_tu = \tilde{\ons}^h_u\left(\frac{\Delta_\pi^hu}{u}+(\Delta_\pi^h)^\star\log(u/\pi^h)\right),
  \end{align}
  which is a discretized version of \eqref{eq:dlss2}.  
\end{rmk}

\subsection{Proof of the main theorem}
We are finally in the position to prove Theorem \ref{thm:main}, i.e., 
we derive the estimates \eqref{eq:ddlssdecay}, \eqref{eq:dauxiliary} and \eqref{eq:L1}
for the gradient flow \eqref{eq:ddlss2} of the discrete Fisher information functional $\fish^h$.

The estimates \eqref{eq:ddlssdecay} and \eqref{eq:dauxiliary} follow easily by means of Proposition \ref{prp:mms}.
Indeed, in order to verify that Proposition \ref{prp:mms} applies in our situation, 
it suffices to observe that the discrete entropy functional $\ent^h$ satisfies \eqref{eq:cvx-e-d},
which is a consequence of Proposition \ref{prp:dfokkerconvex} above,
and of the fact that $\fish^h=|\partial\ent^h|^2$ by definition in \eqref{eq:dmiracle2}.
For the proof of \eqref{eq:L1}, we combine the first estimate in \eqref{eq:dlssdecay} 
with the Csiszar-Kullback inequality, see e.g. \cite{csiszar},
which specializes in the case at hand to 
\begin{align*}
  \| u^h - \pi^h \|_{L^1(\cube)}^2 \le 2\ent^h(u^h). 
\end{align*}

\subsection{Discretization in time}
We shall now use our spatial discretization as basis for the implementation of a numerical scheme
for approximate solution of \eqref{eq:dlss}.
More precisely, we apply a discretization in time to the ordinary differential equations \eqref{eq:ddlss2},
\begin{align*}
  \frac{\dd}{\dd t} U = F^h(U) = K^h_US_U, 
  \qtextq{with} 
  (S_U)_\bi = \frac{(\mc^hU)_\bi}{U_\bi} + \left[(\mc^h)^T\log\left(\frac{U_\bj}{\Pi^h_\bj}\right)_{\bj\in J^h}\right]_\bi.
\end{align*}
For discretization in time, an implicit Euler scheme is employed:
we replace the function $U^h:[0,T]\to\setR_+^{J^h}$ by a sequence $(U^{h,\tau}_m)_m$ 
with the interpretation that $U^{h,\tau}_m$ approximates $U^h(m\tau)$,
and solve
\begin{align}
  \label{eq:euler}
  \frac{U^{h,\tau}_m-U^{h,\tau}_{m-1}}\tau = F^h(U^{h,\tau}_m)
\end{align}
inductively for $m=1,2,\ldots$.
The (first order) implicit Euler method is the canonical choice here 
since it transfers the decay estimates \eqref{eq:ddlssdecay}
from the semi-discrete to the fully discrete level.
\begin{prp}
  \label{prp:eulerdown}
  Assume that a sequence $(U^{h,\tau}_m)_{m\geq0}$ satisfies \eqref{eq:euler}.
  Then the following time-discrete variants of \eqref{eq:ddlssdecay} hold:
  \begin{align}
    \label{eq:dddlssdown}
    \ent^h(U^{h,\tau}_m) \le \ent^h(U^{h,\tau}_{m'})\,\big(1+(2\lambda^h)^2\tau\big)^{-(m-m')}
    \qtextq{and}
    \fish^h(U^{h,\tau}_m) \le \fish^h(U^{h,\tau}_{m'})\,\big(1+(2\lambda^h)^2\tau\big)^{-(m-m')},
  \end{align}
  for all integers $m\ge m'\ge0$.
\end{prp}
The proof of Proposition \ref{prp:eulerdown} is a consequence of the following convexity property.
\begin{lem}
  \label{lem:allconvex}
  Both $\ent^h$ and $\fish^h$ are convex on $\prb(J^h)$ in the sense of linear interpolation.
\end{lem}
\begin{rmk}
  We emphasize that convexity with respect to linear interpolation 
  and geodesic convexity with respect to the Onsager operator $\ons^h$
  are (almost) unrelated notions.
\end{rmk}
\begin{proof}[Proof of Lemma \ref{lem:allconvex}]
  First, recall that $\phi:\setR_+\to\setR$ with $\phi(s)=s\log s$ has derivatives
  \[ \phi'(s) = 1+\log s, \quad \phi''(s) = 1/s. \]
  Given $U\in\prb(J^h)$ and $\Xi\in\tg_U\prb(J^h)$, 
  we have on the one hand that
  \begin{align*}
    \dff_U^2\ent^h[\Xi]^2
    = h^d\sum_\bj\partial_{U_\bj}^2\left(U_\bj\log U_\bj-U_\bj\log\Pi^h_\bj\right)\Xi_\bj^2
    = h^d\sum_\bj U_\bj^{-1}\Xi_\bj^2 \ge 0,
  \end{align*}
  and on the other hand that
  \begin{align*}
    \dff_U^2\fish^h[\Xi]^2
    &= h^d \sum_{\nbh\bi\bj}\sqrt{\Pi^h_\bi\Pi^h_\bj}\bigg\{
    \partial_{U_\bi}^2\left(\left(\frac{U_\bi}{\Pi^h_\bi}-\frac{U_\bj}{\Pi^h_\bj}\right)\log\frac{U_\bi}{\Pi^h_\bi}\right)\Xi_\bi^2
    + \partial_{U_\bj}^2\left(\left(\frac{U_\bj}{\Pi^h_\bj}-\frac{U_\bi}{\Pi^h_\bi}\right)\log\frac{U_\bj}{\Pi^h_\bj}\right)\Xi_\bj^2 \\
    &\qquad -2\partial_{U_\bi}\partial_{U_\bj}\left(\frac{U_\bi}{\Pi^h_\bi}\log\frac{U_\bj}{U^h_\bj}+\frac{U_\bj}{\Pi^h_\bj}\log\frac{U_\bi}{U^h_\bi}\right)\Xi_\bi\Xi_\bj
    \bigg\}\\
    &= h^d\sum_{\nbh\bi\bj} \sqrt{\Pi^h_\bi\Pi^h_\bj}
    \left(\frac{\Pi^h_\bi}{U_\bi}+\frac{\Pi^h_\bj}{U_\bj}\right)
    \left(\sqrt{\frac{\Pi^h_\bi}{U_\bi}\frac{U_\bj}{\Pi^h_\bj}}\frac{\Xi_\bi}{\Pi^h_\bi}-\sqrt{\frac{\Pi^h_\bj}{U_\bj}\frac{U_\bi}{\Pi^h_\bi}}\frac{\Xi_\bj}{\Pi^h_\bj}\right)^2
    \ge 0.
  \end{align*}
  Non-negativity of the second derivatives implies convexity.
\end{proof}
\begin{proof}[Proof of Proposition \ref{prp:eulerdown}]
  Apply the derivative of $\ent^h$ at $U^{h,\tau}_m$ to \eqref{eq:euler} to obtain
  \begin{align*}
    -\dff_{U^{h,\tau}_m}\ent^h\left[\frac{U^{h,\tau}_m-U^{h,\tau}_{m-1}}\tau\right]
    = \dff_{U^{h,\tau}_m}\ent^h\left[\ons_{U^{h,\tau}_m}\dff_{U^{h,\tau}_m}\fish^h\right]
    \ge (2\lambda_h)^2\ent^h(U^{h,\tau}_m),
  \end{align*}
  where we have used the estimate \eqref{eq:interm} and \eqref{eq:cvx-e-d} with constant $\lambda^h$   to obtain the inequality.
  Furthermore, since $\ent^h$ is convex by Lemma \ref{lem:allconvex} above,
  \begin{align*}
    \ent^h(U^{h,\tau}_{m-1})\ge \ent^h(U^{h,\tau}_m) - \dff_{U^{h,\tau}}\ent^h\left[U^{h,\tau}_m-U^{h,\tau}_{m-1}\right]
    \ge \left(1+(2\lambda^h)^2\tau\right)\ent^h(U^{h,\tau}_m).
  \end{align*}
  An iteration of this estimate yields the first inequality in \eqref{eq:dddlssdown}.
  The proof of the second inequality is obtained in an analogous way,
  now applying the derivative of $\fish^h$ in place of $\ent^h$ to \eqref{eq:euler},
  using the estimate \eqref{eq:interm-2}, and the convexity of $\fish^h$ with respect to linear interpolation.
\end{proof}

\subsection{Numerical experiments}
In our experiments, we restrict attention to the two-dimensional situation $d=2$, 
i.e., $\cube=[0,1]^2$ is the unit square.
For the potential, we have used $V(x)=\lambda/2|x-\bar x|^2$, with $\bar x=(1/2,1/2)$ the center of $[0,1]^2$,
corresponding to $W(x) = \lambda^2|x-\bar x|^2-4\lambda$.
Different choices for the convexity parameter $\lambda$ are used in the simulations.
In each experiment, a spatial resolution of $N=30$ grid points in each direction has been used.
The time step $\tau>0$ is chosen in dependence of $\lambda$;
since we solve the implicit Euler scheme \eqref{eq:euler} by an undamped Newton iteration in each time step,
a sufficiently small $\tau$ is necessary for numerical well-posedness of the scheme.

\subsubsection{Illustration of qualitative behavior}
For illustration of the complex qualitative behavior of solutions $u$ to \eqref{eq:dlss},
we report results for a numerical experiment 
in the unconfined case $W\equiv0$, i.e., $\lambda=0$,
for the initial datum
\begin{align}
  \label{eq:BLSu0}
  u_0(x) = \frac1Z\big(\cos^{16}\pi x_1 + \cos^{16}\pi x_2\big) + 10^{-4},
\end{align}
where $Z=0.392\ldots$ is such that $u_0$ integrates to one on $[0,1]^2$.
The initial condition is drawn in Figure \ref{fig:BLSinit}.
This is a straight-forward generalization of the one-dimensional example from \cite[Figure 1]{BLS} to two space dimensions.
Notice that $u_0$ has a large plateau where its values are very small (order $10^{-4}$) 
in comparison to the average value (order one).
\begin{figure}
  \centering
  \includegraphics[width=0.3\textwidth]{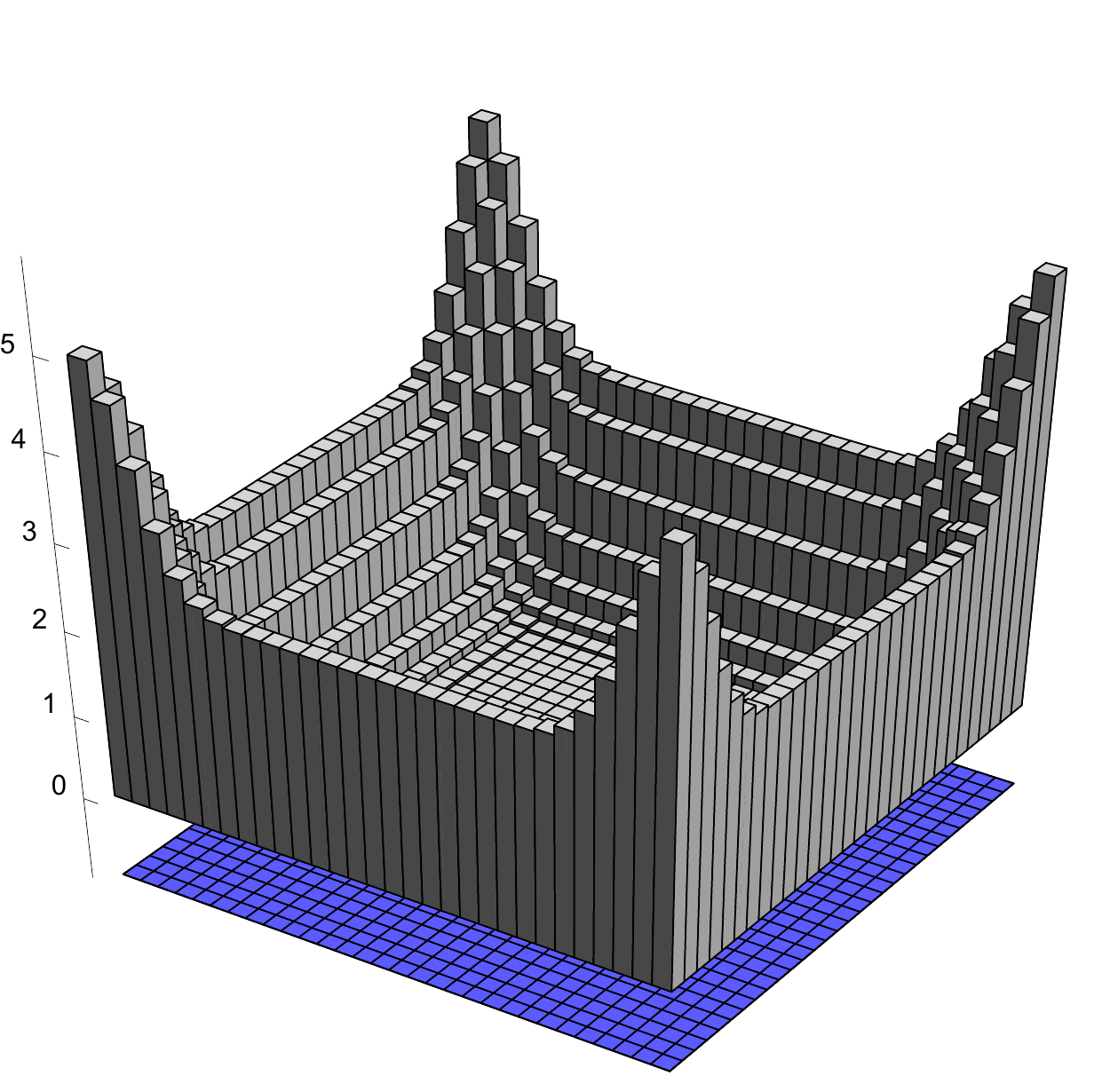}
  \includegraphics[width=0.3\textwidth]{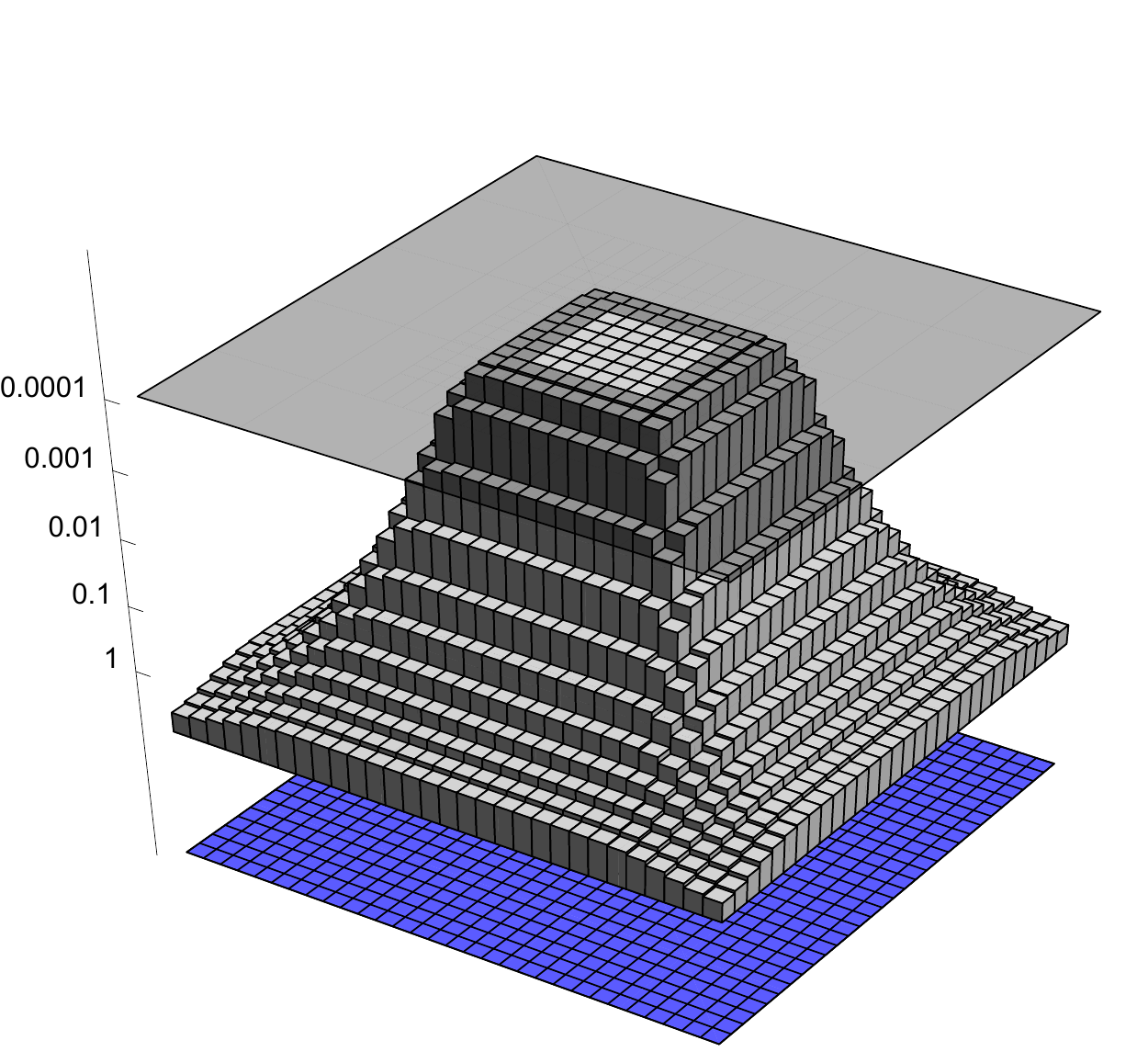}
  \includegraphics[width=0.3\textwidth]{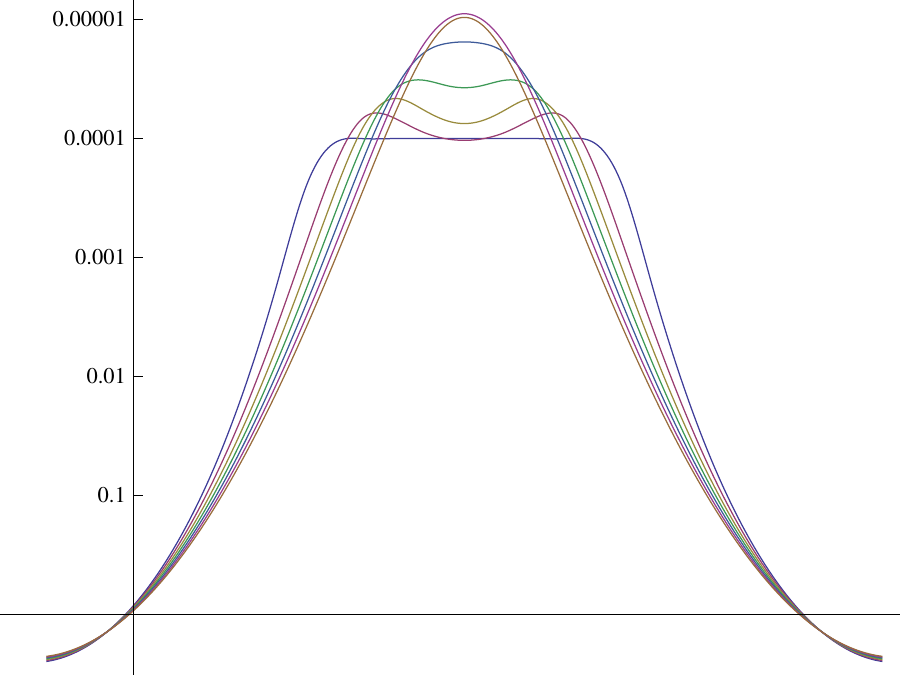}
  \caption{Left and middle: initial condition $u_0$ from \eqref{eq:BLSu0} in normal and in logarithmic scale.
    Right: values of the corresponding solution on one-dimensional the cross-section $x_2=1/2$
    at times $t=k\times10^{-6}$ for $k=0,\,0.5,\,1.0,\,1.5,\,2.0,\,2.5,\,3.0$.}
  \label{fig:BLSinit}
\end{figure}
\begin{figure}
  \centering
  \includegraphics[width=0.3\textwidth]{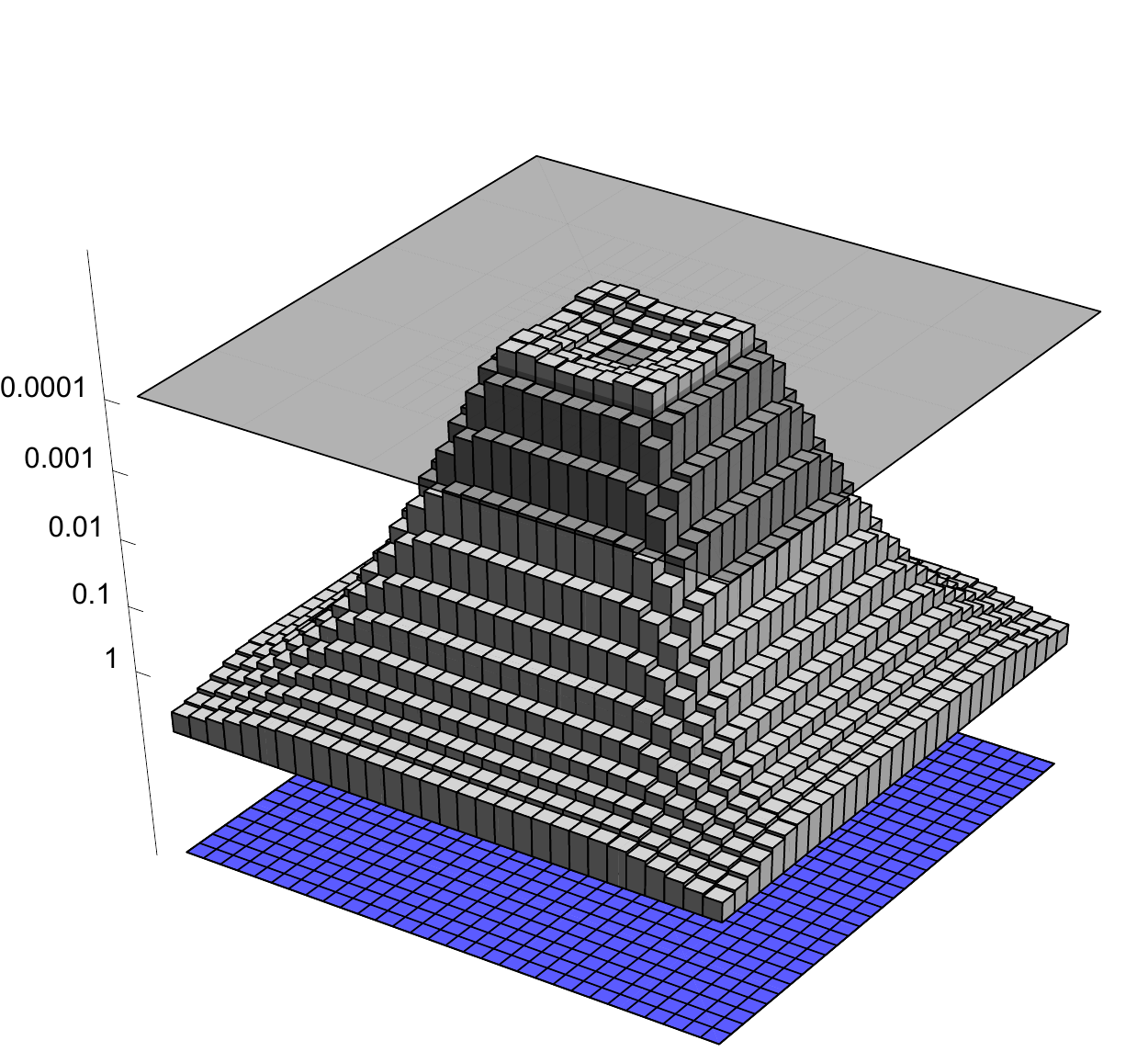}
  \includegraphics[width=0.3\textwidth]{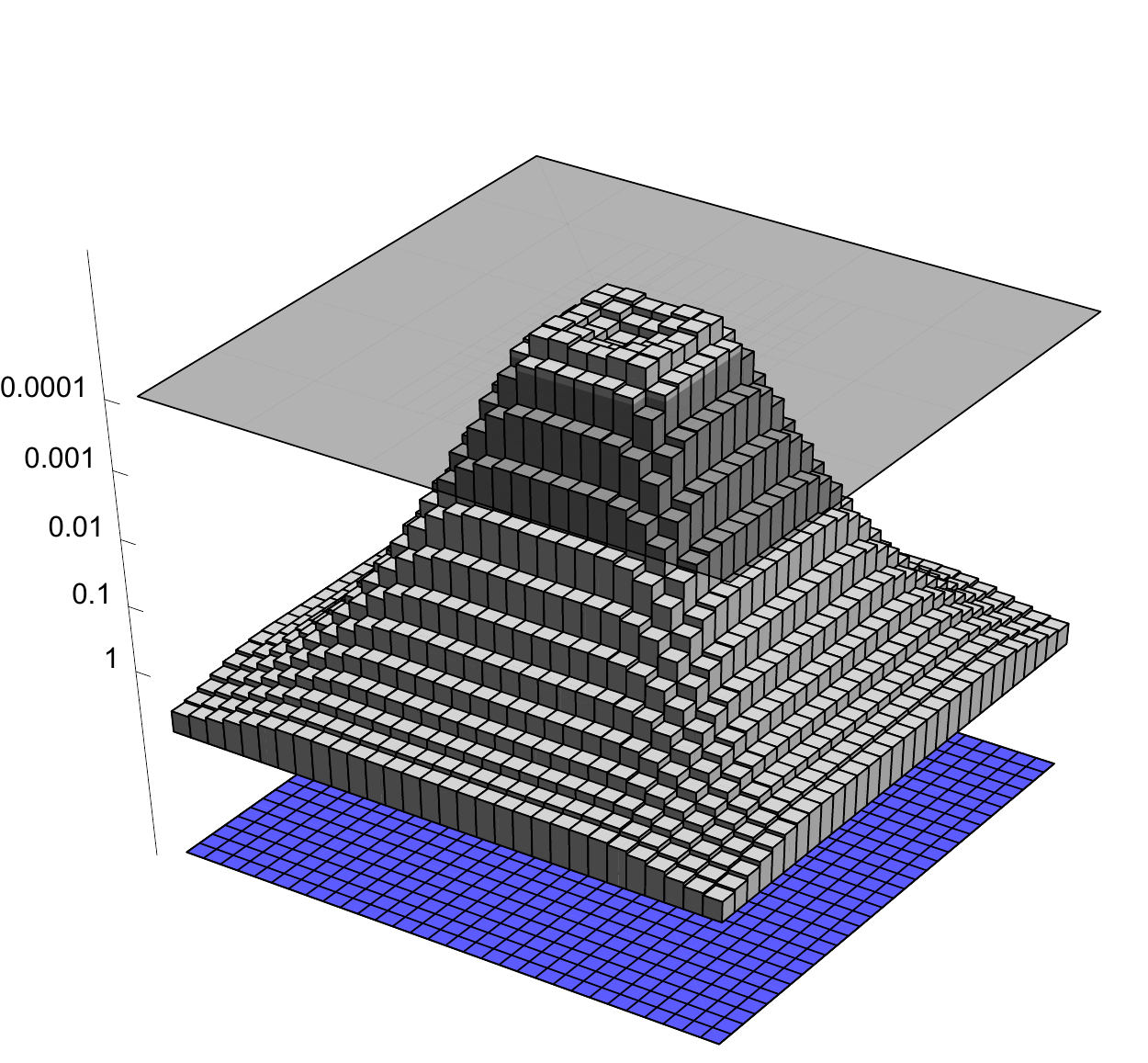}
  \includegraphics[width=0.3\textwidth]{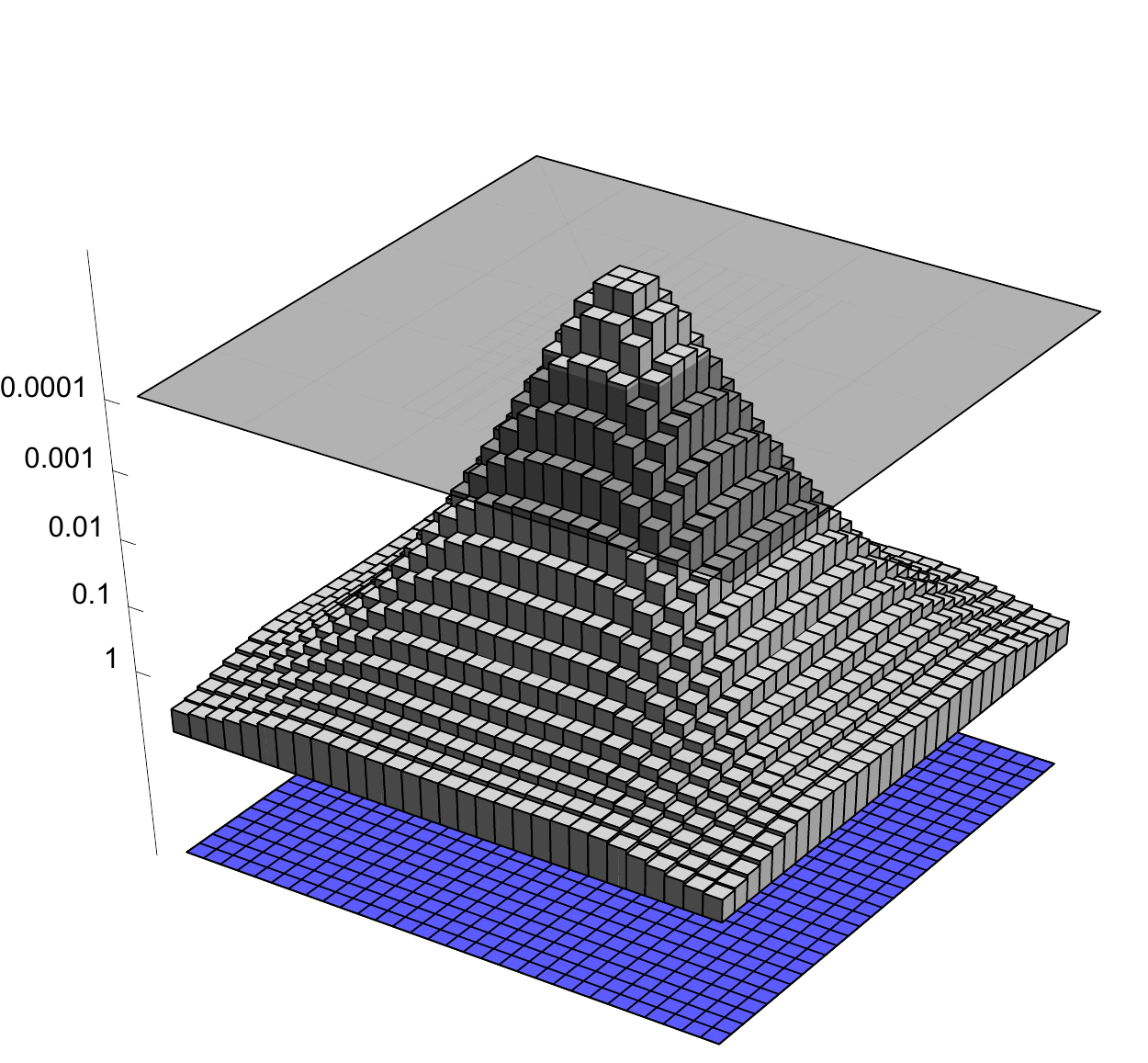}
  \caption{Behavior of the solution for the initial condition $u_0$ from \eqref{eq:BLSu0} 
    at times $t=0.3\times10^{-6}$, $t=1.0\times10^{-6}$, and $t=3.0\times10^{-6}$ (left to right),
    plotted in logarithmic scale.
    The transparent plane at $u\equiv10^{-4}$ has been introduced to visualize 
    that the solution does not obey a maximum principle.}
  \label{fig:BLS3d}
\end{figure}
The sharp flanks at the edge of the plateau drive the dynamics 
and lead to a rather complicated spatial-temporal behavior of the solution, see Figure \ref{fig:BLS3d}.
The right of Figure \ref{fig:BLSinit} shows a one-dimensional cross-section of the solution;
qualitatively, the behavior is in perfect agreement with the one-dimensional simulations from \cite{BLS}.
A time step $\tau=10^{-7}$ has been used for the numerical solution 
in order to resolve the process of creation and destruction of local minima inside the plateau region, 
which happens on a time scale of $10^{-6}$.

\subsubsection{Rates of equilibration}
The goal of the following series of experiments is 
the numerical verification of the analytically estimated rates of equilibration.
We vary the convexity parameter $\lambda$ and 
apply the numerical scheme to the very regular initial condition
\begin{align}
  \label{eq:u0}
  u_0(x) = \frac43\sin^2(3\pi x_1)\sin^2(2\pi x_2)+ \frac13(1+x_1+x_2).
\end{align}
The behavior of entropy and Fisher information are monitored for about one thousand time iterations.
The qualitative change in density is shown in Figure \ref{fig:equilibration}.
The corresponding results for entropy and Fisher information are collected in Figure \ref{fig:entfish}.

In agreement with the analytical estimates in \eqref{eq:dddlssdown}, 
both quantities decay with a rate of at least $(2\lambda^h)^2$.
In fact, in each experiment we measure a minimal decay rate $(2\lambda^h_*)^2$ 
that is strictly larger than the analytically predicted rate.
Generally, the difference $\lambda^h_*-\lambda^h$ is the larger the smaller $\lambda>0$ is,
and becomes negligible for large values $\lambda\gg10$.

This phenomenon is apparently independent of the spatial resolution $h>0$.
Our conjecture is the following.
For solutions to the Fokker-Planck equation \eqref{eq:fokker}, the estimates \eqref{eq:fokkerdecay} are not sharp:
the lower bound on the rate of equilibration is given by $2\lambda_*$ with some $\lambda_*>\lambda$.
This improvement is due to boundary effects:
it is neglegible if the steady state $\pi$ is very concentrated inside $\cube$ (as is the case for $\lambda\gg10$),
but is significant for more equally distributed stationary densities $\pi$ (for $\lambda<10$ or less).
Thanks to its intimate relation to the Fokker-Planck equation \eqref{eq:fokker},
the fourth order equation \eqref{eq:dlss} apparently inherits these improved rates, i.e.,
one can replace $(2\lambda)^2$ by $(2\lambda_*)^2$ in the estimates \eqref{eq:dlssdecay}.
These improved estimates pass on to the estimates \eqref{eq:ddlssdecay} and \eqref{eq:dddlssdown} on the discretization.

Our conjecture is strongly supported by the outcome of the experiments.
In Figure \eqref{fig:entfish}, the decay rates of entropy and Fisher information are compared to $(2\lambda^h_*)^2$,
where $\lambda^h_*>\lambda^h$ is the smallest non-zero eigenvalue of the associated Markov generator $\mc^h$ on $J^h$.
In all of the experiments that have been performed, the numerically measured rate of decay of entropy and Fisher information,
\begin{align*}
  \frac1\tau\left(\log\ent^h(U^{h,\tau}_m)-\log\ent^h(U^{h,\tau}_{m+1})\right)
  \qtextq{and}
  \frac1\tau\left(\log\fish^h(U^{h,\tau}_m)-\log\fish^h(U^{h,\tau}_{m+1})\right),
  \qtext{respectively,}
\end{align*}
never fall below the value $(2\lambda^h_*)^2$.
In fact, the numerically measured rates have always been larger 
but appear to tend towards $(2\lambda^h_*)^2$ as the system approaches equilibrium.
This is in accordance with the observation from \cite{MMS} that the equilibration rates are minimized 
in the linearized regime around the steady state.
\begin{rmk}
  It is tempting to turn the above conjecture into a proof, 
  simply using the spectral gap $\lambda_*$ instead of $\lambda$ as a lower bound on the modulus of geodesic convexity of $\ent$
  and performing all the estimates accordingly.
  However, to our knowledge, there is no result available which allows to estimate
  the modulus of geodesic convexity of a Markov chain 
  --- or the corresponding constant $\lambda$ in the inequality \eqref{eq:cvx-e-d} ---
  by its spectral gap from below.
\end{rmk}
\begin{figure}
  \centering
  \includegraphics[width=0.3\textwidth]{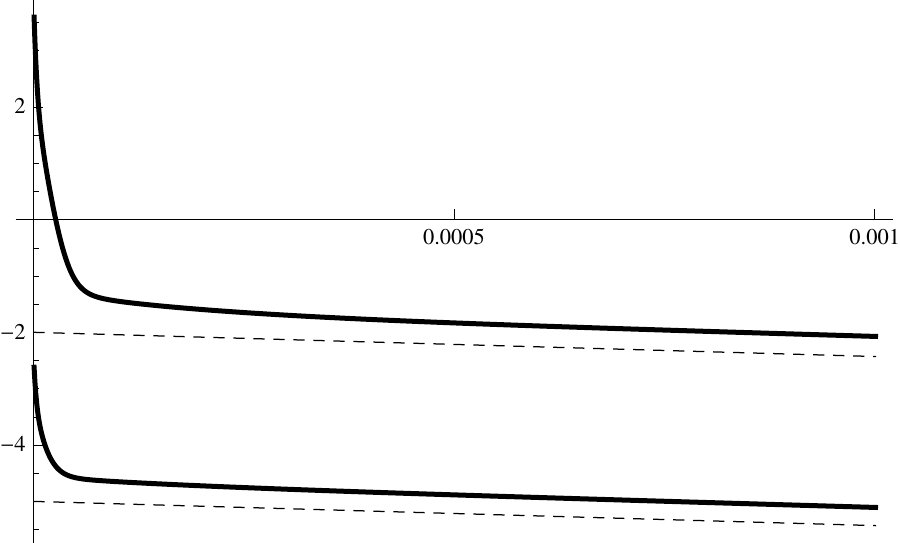}
  \includegraphics[width=0.3\textwidth]{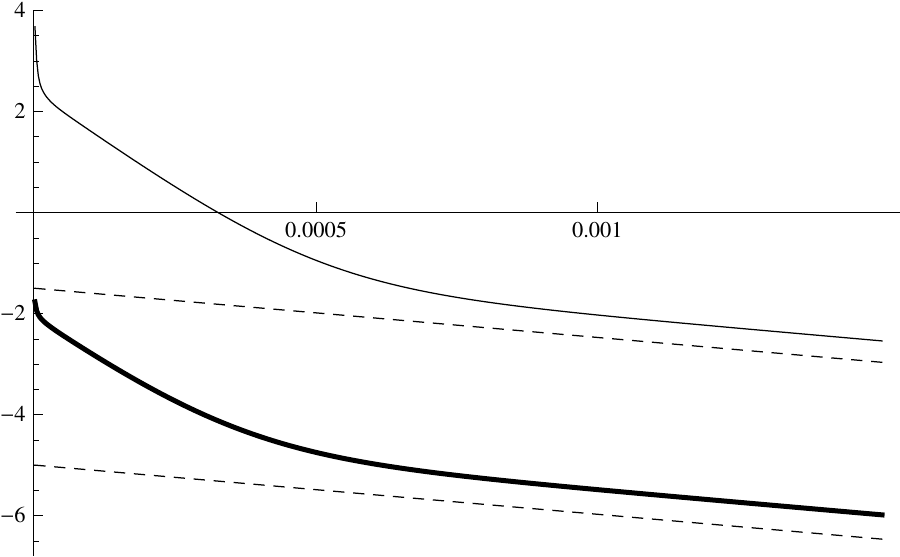}
  \includegraphics[width=0.3\textwidth]{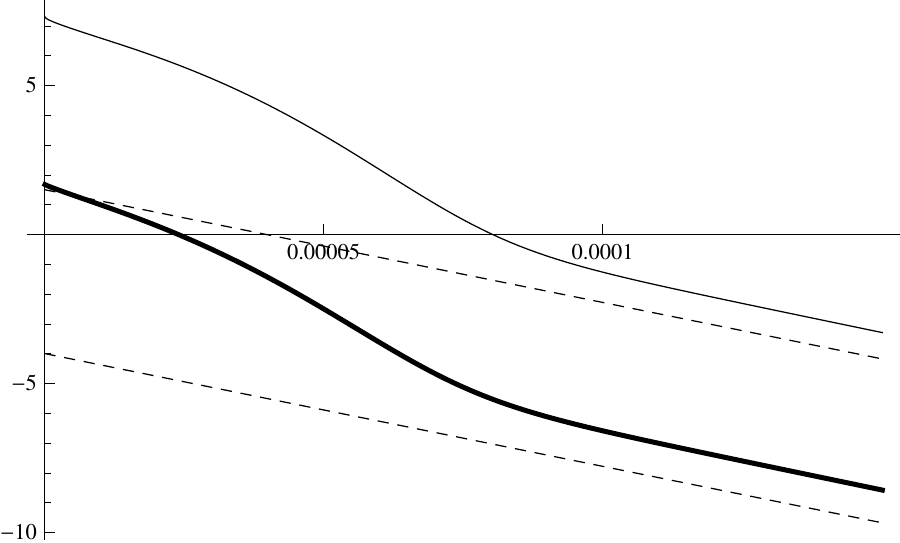}
  \caption{Logarithmic plot of entropy (lower bold curve) and Fisher information (upper bold curve) 
    along discrete solutions for the initial condition \eqref{eq:u0},
    using $\lambda=1$, $\lambda=10$ and $\lambda=100$ (left to right).
    The dotted lines correspond to multiples of $\exp(-(2\lambda^h_*)^2t)$.}
  \label{fig:entfish}
\end{figure}
\begin{figure}
  \centering
  \includegraphics[width=0.3\textwidth]{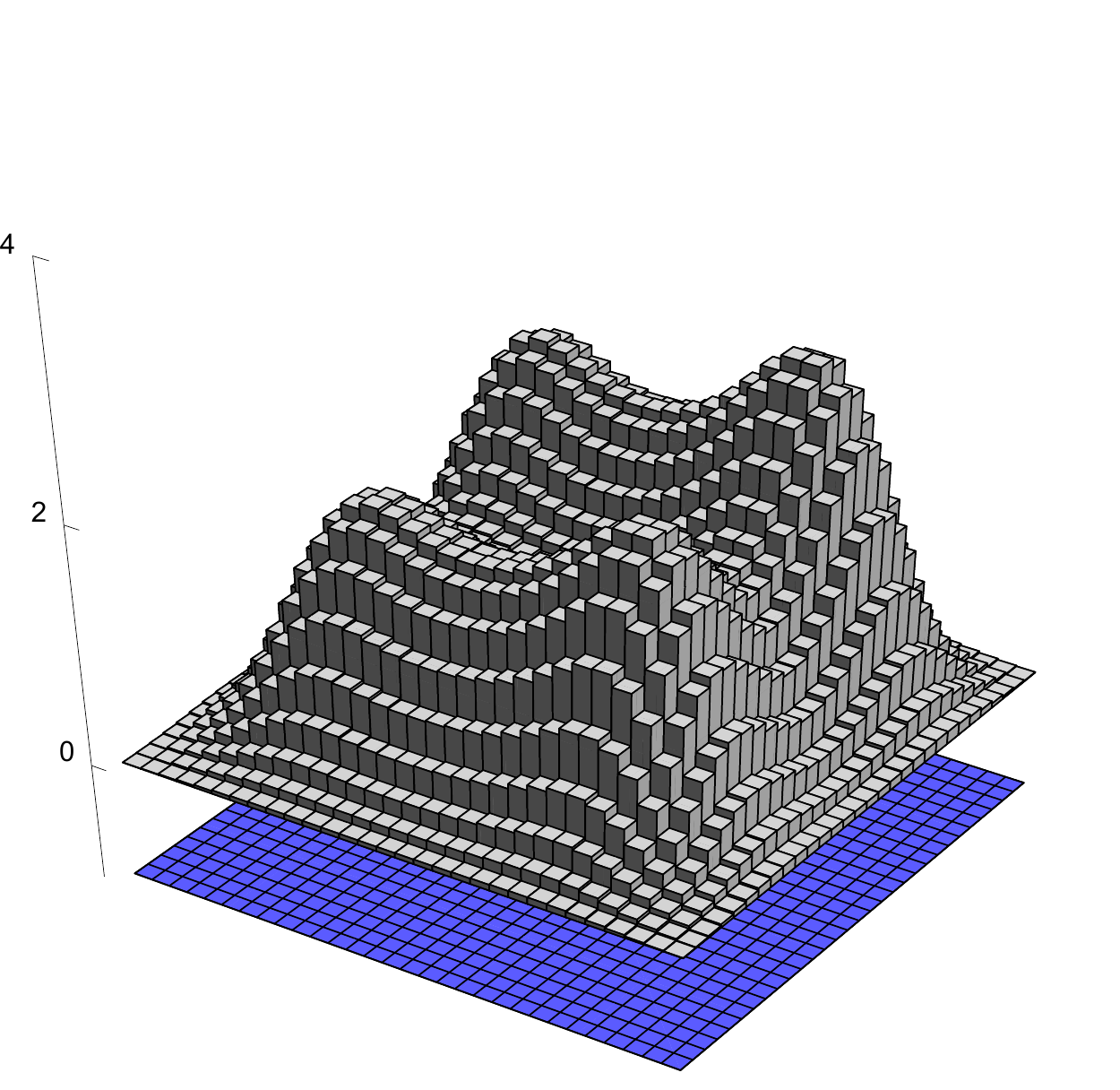}
  \includegraphics[width=0.3\textwidth]{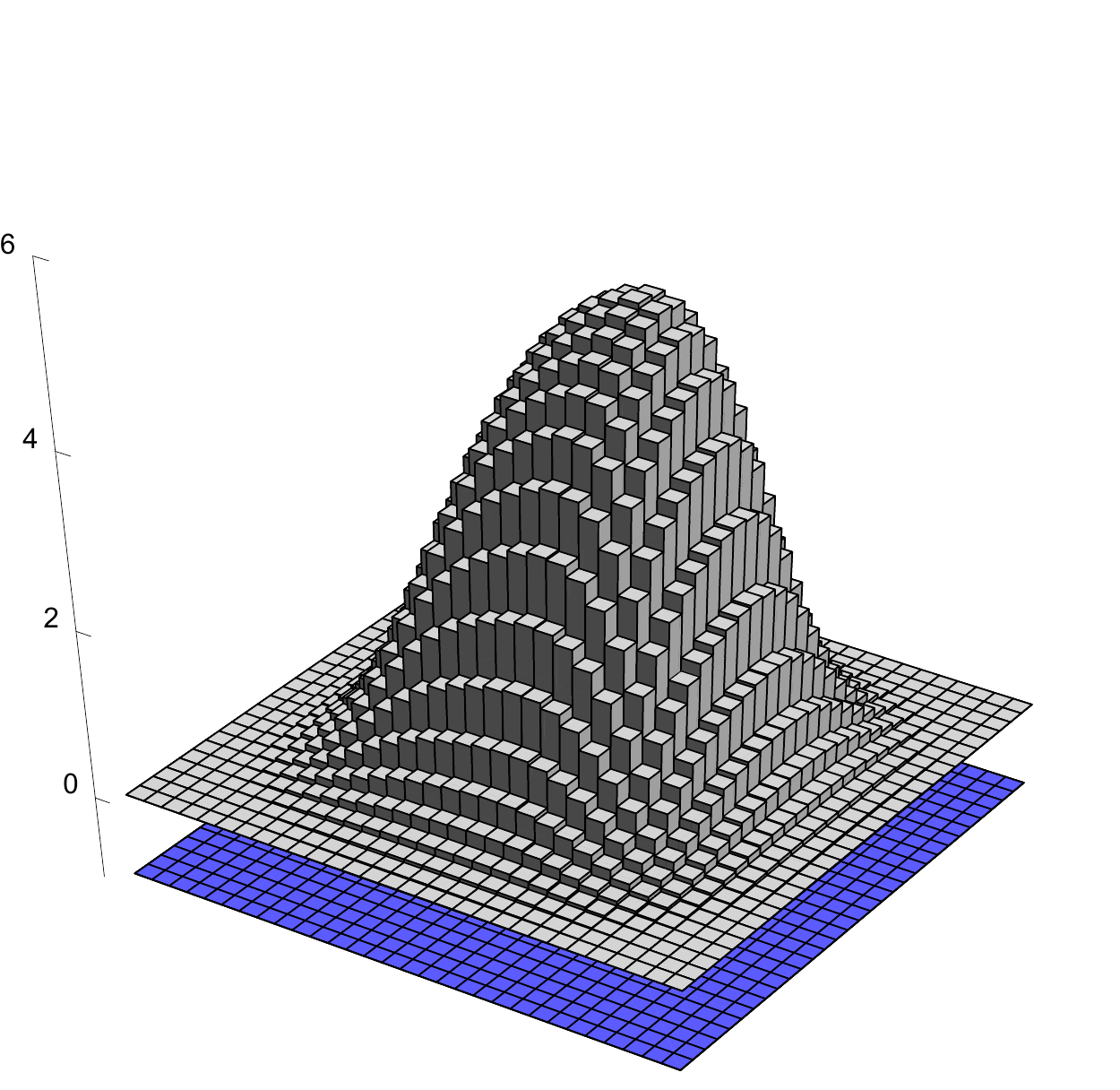}
  \includegraphics[width=0.3\textwidth]{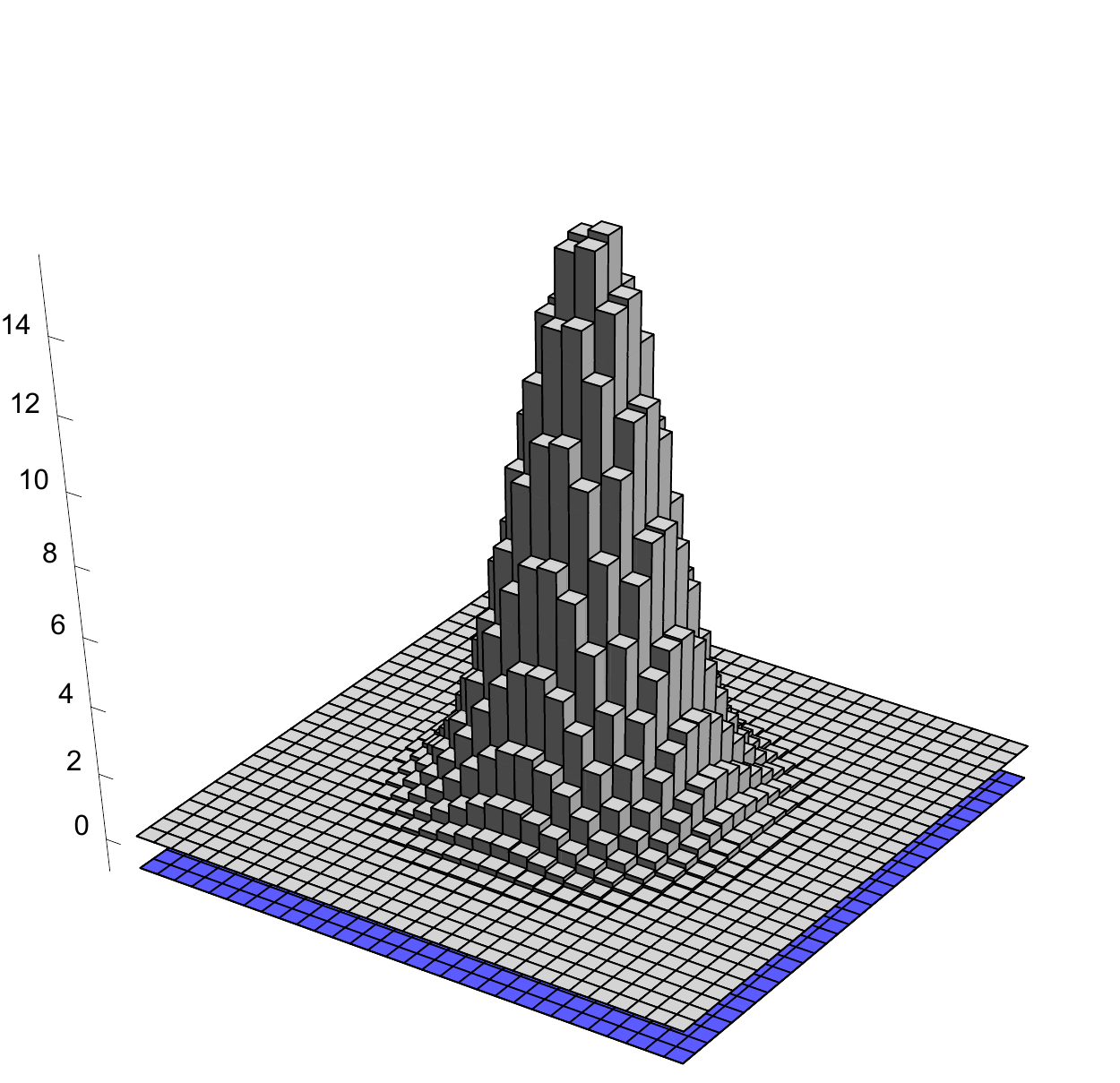}  
  \caption{Snapshots of the discrete solution for the initial condition \eqref{eq:u0}
    using $\lambda=100$ at, respectively, $t=10^{-5}$, $t=3\cdot10^{-5}$ and $t=10^{-4}$ (left to right).
    Note the changing scale.}
  \label{fig:equilibration}
\end{figure}

\bibliography{dlss}

\bibliographystyle{plain}

\end{document}